\documentclass[11pt]{article}

\usepackage[letterpaper,margin=0.9in]{geometry}

\usepackage{lipsum}
\usepackage{amsfonts}
\usepackage{graphicx}
\usepackage{epstopdf,hyperref}
\usepackage{mathtools}
\usepackage{stackrel}
\usepackage{stmaryrd}
\usepackage{amsmath}
\usepackage{subfig}
\usepackage{amsmath,amsthm}
\usepackage{listings}
\usepackage{algorithmic}
\ifpdf
  \DeclareGraphicsExtensions{.eps,.pdf,.png,.jpg}
\else
  \DeclareGraphicsExtensions{.eps}
\fi

\usepackage{tikz}
\usetikzlibrary{positioning,chains,fit,shapes,calc}
\definecolor{grey1}{rgb}{0.0,0.0,0.0}
\definecolor{grey2}{rgb}{0.2,0.2,0.2}
\definecolor{grey3}{rgb}{0.4,0.4,0.4}
\definecolor{grey4}{rgb}{0.5,0.5,0.5}

\definecolor{myblue}{RGB}{100,100,160}
\definecolor{mygreen}{RGB}{80,160,80}

\tikzset{
  -|-/.style={
    to path={
      (\tikztostart) -| ($(\tikztostart)!#1!(\tikztotarget)$) |- (\tikztotarget)
      \tikztonodes
    }
  },
  -|-/.default=0.5,
  |-|/.style={
    to path={
      (\tikztostart) |- ($(\tikztostart)!#1!(\tikztotarget)$) -| (\tikztotarget)
      \tikztonodes
    }
  },
  |-|/.default=0.5,
}

\tikzset{
  font={\fontsize{9pt}{6}\selectfont}}
  
\tikzstyle{startstop1} = [rectangle, rounded corners, minimum width=5cm, minimum height=1cm,text centered, draw=black, fill=red!30]
\tikzstyle{startstop2} = [rectangle, rounded corners, minimum width=5cm, minimum height=1cm,text centered, draw=black, fill=blue!30]
\tikzstyle{startstop3} = [rectangle, rounded corners, minimum width=5cm, minimum height=1cm,text centered, draw=black, fill=orange!30]
\tikzstyle{io} = [trapezium, trapezium left angle=70, trapezium right angle=110, minimum width=3cm, minimum height=0.5cm, text centered, draw=black, fill=blue!30]
\tikzstyle{process} = [rectangle, minimum width=3cm, minimum height=1cm, text centered, draw=black, fill=orange!30]
\tikzstyle{process-phi} = [rectangle, minimum width=3cm, minimum height=1cm, text centered, draw=black, fill=blue!30]
\tikzstyle{process-psi} = [rectangle, minimum width=3cm, minimum height=1cm, text centered, draw=black, fill=green!30]
\tikzstyle{process-lambda} = [rectangle, minimum width=3cm, minimum height=1cm, text centered, draw=black, fill=orange!30]
\tikzstyle{decision} = [diamond, minimum width=3cm, minimum height=0.5cm, text centered, draw=black, fill=red!30]
\tikzstyle{arrow} = [thick,->,>=stealth]

\theoremstyle{plain}
\newtheorem{thm}{Theorem}[section]
\newtheorem{corollary}[thm]{Corollary}
\newtheorem{lemma}[thm]{Lemma}
\newtheorem{proposition}[thm]{Proposition}

\theoremstyle{definition}

\lstdefinelanguage{coco}[]{Matlab}{
	keywords={},
	morekeywords={},
	morekeywords=[2]{classdef,properties,private,protected,public,%
Access,Static,methods,if,function,end,for,while,else,elseif,switch,%
case,otherwise,do,repeat,until},
	otherkeywords={{,end},:end,(end,end),\{end,end\},[end,end],/private,private/},
	morecomment=[l]{},
	basicstyle=\ttfamily\mdseries\footnotesize,
	commentstyle=\ttfamily\mdseries\footnotesize,
	moredelim=[l][\color{\hiddenA}\ttfamily\mdseries\footnotesize]{\%\#},
	moredelim=[l][\color{\hiddenB}\ttfamily\mdseries\footnotesize]{\%!},
	keywordstyle=\ttfamily\mdseries\footnotesize,
	keywordstyle={[2]\ttfamily\mdseries\footnotesize},
	numbers=none,
	numberstyle=\scriptsize,
	numbersep=1em,
	breaklines=false,
	breakatwhitespace=true,
	breakindent=2.5em,
	showlines=false,
	lineskip=-0.2ex,
	frame=none,
	fontadjust=true,
	columns=[c]fixed,
	basewidth={0.575em,0.45em},
	fontadjust=true,
	tabsize=3,
	showstringspaces=false,
	aboveskip=1.5\medskipamount,
	belowskip=1.5\medskipamount,
	xleftmargin=0.925em,
	xrightmargin=0em,
	rangeprefix={\%!},
	includerangemarker=false,
	belowcaptionskip=\bigskipamount
}

\lstdefinelanguage{coco-highlight-fonts}[]{coco}{
	style=highlight-fonts,
}

\lstdefinelanguage{coco-highlight-colors}[]{coco}{
	style=highlight-colors,
}

\lstdefinelanguage{coco-highlight}[]{coco-highlight-colors}{}

\DeclareMathOperator{\Ima}{Im}

\title{Adjoint-Based Projections for Uncertainty Quantification near Stochastically Perturbed Limit Cycles and Tori\thanks{A shortened treatment of the covariance problem for limit cycles, including the first example in Section 2.2, was published previously in \cite{10.1115/DETC2022-91153}. Additionally, the first author's doctoral dissertation (\url{https://www.ideals.illinois.edu/items/124667}) contains an early draft of the content of this manuscript.} }

\author{Zaid Ahsan\thanks{Department of Mechanical Science and Engineering, University of Illinois at Urbana-Champaign, Urbana, IL 61801 
  (email: danko@illinois.edu)}.
\and Harry Dankowicz\footnotemark[2] \and Christian Kuehn\thanks{Faculty of Mathematics, Technical University of Munich, Boltzmannstr. 3, 85748 Garching bei M\"{u}nchen, Germany 
  (email: ckuehn@ma.tum.de)}}

\usepackage{amsopn}

\begin{document}

\maketitle

\begin{abstract}
This paper presents a new boundary-value problem formulation for quantifying uncertainty induced by the presence of small Brownian noise near transversally stable periodic orbits (limit cycles) and quasiperiodic invariant tori of the deterministic dynamical systems obtained in the absence of noise. The formulation uses adjoints to construct a continuous family of transversal hyperplanes that are invariant under the linearized deterministic flow near the limit cycle or quasiperiodic invariant torus. The intersections with each hyperplane of stochastic trajectories that remain near the deterministic cycle or torus over intermediate times may be approximated by a Gaussian distribution whose covariance matrix can be obtained from the solution to the corresponding boundary-value problem. In the case of limit cycles, the analysis improves upon results in the literature through the explicit use of state-space projections, transversality constraints, and symmetry-breaking parameters that ensure uniqueness of the solution despite the lack of hyperbolicity along the limit cycle. These same innovations are then generalized to the case of a quasiperiodic invariant torus of arbitrary dimension. In each case, a closed-form solution to the covariance boundary-value problem is found in terms of a convergent series. The methodology is validated against the results of numerical integration for two examples of stochastically perturbed limit cycles and one example of a stochastically perturbed two-dimensional quasiperiodic invariant torus in $\mathbb{R}^2$, $\mathbb{R}^2\times S^1$, and $\mathbb{R}^2\times S^1$, respectively, for which explicit expressions may be found for the associated covariance functions using the proposed series solutions. Finally, an implementation of the covariance boundary-value problem in the numerical continuation package \textsc{coco} is applied to analyze the small-noise limit near a two-dimensional quasiperiodic invariant torus in a nonlinear deterministic dynamical system in $\mathbb{R}^4$ that does not support closed-form analysis. Excellent agreement with numerical evidence from stochastic time integration shows the potential for using deterministic continuation techniques to study the influence of stochastic perturbations for both autonomous and periodically excited deterministic vector fields.
\end{abstract}

\section{Introduction}
Asymptotically stable equilibria~\cite{Bashkirtseva2017,fan2011covariances,MENA2019146}, limit cycles~\cite{berglund2006noise,cheng2021stochastic,kurrer1991effect}, invariant curves~\cite{BASHKIRTSEVA2014236}, and transversally stable quasiperiodic invariant tori~\cite{bashkirtseva2016sensitivity,bashkirtseva2020stochastic,Guo2017} of deterministic dynamical systems continue to organize the local flow over intermediate times also with the addition of noise of small intensity. There results metastable dynamics near the deterministic limit sets with stochastic trajectories characterized over such time scales by approximately stationary Gaussian probability distributions of intersections with hyperplanes transversal to the local vector field. Knowledge of these distributions allows the investigator to predict the noise-induced spread of uncertainty about the deterministic limit set~\cite{Bashkirtseva2017,biswas2021characterising,Gao202313513,Garain2022} and the likelihood of transitions between different metastable objects~\cite{Guo2020599,Huang20234219,Jungeilges20215849,safonov2006noise,Salman202269,tadokoro2020noise,Wang2021135}.

Given a stochastic differential equation that appends noisy excitation to an underlying deterministic dynamical system, numerical integration of solution trajectories is a commonly deployed brute-force approach for the characterization of noise-induced uncertainty, especially for global analysis \cite{higham2021introduction}. While computationally expensive and necessarily limited to particular choices of parameter values and noise models, the results of such analysis provide an important reality check for theoretical predictions~\cite{biswas2021characterising}. For local analysis---of primary interest here---a more affordable alternative is the direct computation of covariance ellipsoids~\cite{bashkirtseva2004stochastic,ryashko2009confidence} in each hyperplane through the solution of an associated Lyapunov equation~\cite{bittanti2009periodic,bolzern1988periodic,Halanay1987} for a covariance tensor function. Indeed, as shown by Kuehn~\cite{kuehn2012deterministic} (see also \cite{baars2017continuation,kuehn2015efficient,kuehn2015numerical}) in the case of equilibria, such a computation may be embedded within a numerical continuation framework in order to investigate the parameter dependence of the metastable dynamics and the probabilities of noise-induced transitions between metastable objects (see, e.g., \cite{bashkirtseva2015stochastic}). One goal of this paper is to extend such a methodology, compatible with a numerical parameter continuation framework, also to the case of limit cycles and transversally stable quasiperiodic invariant tori.

To achieve such an outcome, we expand on our preliminary treatment of limit cycles in \cite{10.1115/DETC2022-91153} to formulate a smooth, regular boundary-value problem, such that pairings of parameter values, deterministic limit sets, and associated covariance tensor functions span a smooth manifold in a suitably defined variable space. Such a formulation then supports simultaneous continuation of limit sets and covariance functions using well-established software packages, e.g., \textsc{auto}~\cite{doedel2007auto} and \textsc{coco}~\cite{dankowicz2013recipes}, including the possibility of monitoring or constraining the geometry of the covariance ellipsoids during parameter variations. In this paper, we rely on an implementation in \textsc{coco} to enable numerical validation of the theoretical predictions, and make the corresponding code freely available at \url{https://github.com/hdankowicz/covariance-bvp2023-scripts}. Although we leave the utility of such a combined framework largely to future work, a natural application is to problems in optimal robust design of noise-disturbed engineering systems.

Several features distinguish our treatment from that in most of the existing literature. Foremost among these is the construction of transversal hyperplanes using projection operators defined in terms of the solutions to appropriately defined adjoint boundary-value problems with integral constraints as described in \cite{Dankowicz2022329}. Notably, this construction is invariant under local coordinate transformations. The resultant foliation of hyperplanes is also invariant under the linearized flow near the limit set, which contracts the corresponding transversal deviations exponentially fast to $0$. Neither of these properties hold for the typical construction (for autonomous systems) of transversal hyperplanes that are chosen arbitrarily~\cite{MANCHESTER20116285,Zhao2022} or required to be orthogonal to the local vector field~\cite{bashkirtseva2016sensitivity,louca2018stable}, as this construction is not invariant under local coordinate transformations, the resultant hyperplanes are not invariant under the linearized flow, and deviations tangential to the limit set must be eliminated, for example, by projection onto a local set of coordinates in each hyperplane (as in \cite{louca2018stable} and \cite{Zhao2022}; see also \cite{Giacomin20181019} for an analysis of phase diffusion along a limit cycle). Such a coordinate projection is not required by the proposed construction which proceeds to make predictions in the original coordinates. Furthermore, as is demonstrated here through several examples, for non-autonomous, periodically excited dynamical systems, the proposed construction automatically generates hyperplanes in stroboscopic sections (the \emph{a priori} choice in~\cite{biswas2021characterising}). Finally, in our computational implementation, we benefit from the existence of automated support for the construction of the adjoint boundary-value problems using the staged construction paradigm in \textsc{coco} \cite{li2018staged}.

As a consequence of these choices, we obtain a Lyapunov equation that conserves the scalar-valued image of the covariance tensor function on any pair of adjoint vectors. We restrict attention to the intersection of the zero-level sets of all such images by appending a corresponding set of boundary conditions and augmenting the Lyapunov equation with a regularizing damping term that must vanish on solutions to the full covariance boundary-value problem (see \cite{Galan-Vioque20142705,MUNOZALMARAZ20031} for a general methodology). Using the invariance of the foliation under the linearized flow and exponential contraction, we obtain a unique solution for the covariance tensor function in terms of an exponentially convergent series. The result is a function whose value at each point of the limit set is a positive semi-definite matrix with as many nonzero eigenvalues, generically, as the co-dimension of the limit set.

That the analysis for limit cycles may be extended also to the case of quasiperiodic invariant tori is demonstrated already in~\cite{bashkirtseva2016sensitivity,Guo2017,Zhao2022}. In our formulation, this extension results in several coupled partial differential equations on a cylindrical domain with characteristics parallel to the time axis and with boundary conditions, integral constraints, and regularizing terms enforcing quasiperiodicity,  non-degeneracy, and uniqueness. In our numerical implementation, we rely on a problem discretization in terms of Fourier series in the non-temporal variables and continuous, piecewise-polynomial functions along characteristics, and show how the discretized adjoint integro-differential boundary-value problems may be obtained from a variational formulation, as also implemented in~\textsc{coco}.

The remainder of this paper is arranged as follows. Section~\ref{sec: Periodic Orbits} considers the influence of noise on the dynamics near limit cycles. Rigorous derivations for systems of arbitrary dimension are there validated against the predictions of numerically integrated stochastic trajectories for two examples where the sought covariance functions may be obtained in closed form, viz., an autonomous, deterministic vector field in $\mathbb{R}^2$ obtained from the Hopf normal form and a harmonically excited, damped, linear oscillator with dynamics in $\mathbb{R}^2\times\mathbb{S}^1$. Section~\ref{sec: Quasiperiodic Tori} generalizes the construction of families of local projections to the case of transversally stable quasiperiodic invariant tori of arbitrary dimension in problems of arbitrary dimension and shows that a unique covariance matrix may again be associated with each corresponding hyperplane. Here, validation is delivered first using a two-dimensional torus of a deterministic vector field in $\mathbb{R}^2\times\mathbb{S}^1$, for which closed-form analysis of the covariance function is possible, and then a two-dimensional torus of a deterministic vector field in $\mathbb{R}^4$, which is not amenable to such analysis and for which an approximation must be obtained from a discretization of the governing boundary-value problem using the software package~\textsc{coco}. This latter implementation also illustrates how uncertainty quantification may be appended to or integrated with parameter continuation of limit cycles or quasiperiodic invariant tori. Appendix~\ref{app: Adjoint Conditions and Problem Discretization} collects important results that enable automated problem construction and approximate analysis in \textsc{coco} of the various adjoint boundary-value problems, including a choice of discretization in terms of finite sets of characteristics in the case of quasiperiodic invariant tori that can be used analogously for finding the tori in the first place and analyzing the associated covariance functions. A concluding discussion in Section~\ref{sec: Conclusions} reflects on directions for further study.

\section{Periodic Orbits}
\label{sec: Periodic Orbits}
In this section, we derive the set of adjoint and covariance boundary-value problems (previously stated without derivation in \cite{10.1115/DETC2022-91153}) whose combined solution, for a given limit cycle of an autonomous dynamical system, characterizes the metastable dynamics in the limit of small noise intensity. The theoretical predictions are validated against two examples for which closed-form expressions may be compared to the results of stochastic numerical integration.

\subsection{Theoretical derivations}
Let $\gamma(t)\in\mathbb{R}^n$ be a periodic solution with period $1$ of the smooth dynamical system defined by the ordinary differential equation (ODE)
\begin{equation}
\label{eq:determ}
\frac{\textnormal{d} x}{\textnormal{d} t}=:\dot{x}=Tf(x),\quad\mbox{ where } x:=x(t)\in\mathbb{R}^n,
\end{equation}
for some positive scalar\footnote{We obtain the form of \eqref{eq:determ} from a dynamical system of the form $y'(\eta)=f(y(\eta))$ with periodic orbit $\gamma(\eta/T)$ of period $T$ by letting $t=\eta/T$ and $x(t)=y(Tt)$. For an autonomous vector field, $T$ is not typically known \textit{a priori} and is, instead, solved for simultaneously with $\gamma$ by appending a suitable phase condition to the periodic boundary-value problem in $x$, e.g., $h(x(0))=0$ for some smooth scalar-valued function $h$. Throughout, we omit explicit dependence on problem parameters to reduce the notational complexity.} $T$. Let $X(t)$ denote the solution to the corresponding variational initial-value problem
\begin{equation}
\label{eq:varprobper}
\dot{X}=T\mathrm{D}f(\gamma)X,\quad X(0)=I_n,
\end{equation}
where $\mathrm{D}f$ denotes the Jacobian of $f$ and $I_n$ is the identity matrix of size $n\times n$. 
By periodicity of $\gamma$, it follows that $X(t+1)=X(t)X(1)$. Furthermore, $X^{-1}(t)f(\gamma(t))\equiv f(\gamma(0))$, since equality holds at $t=0$ and differentiation shows that the left-hand side is constant. As a special case, again by periodicity of $\gamma$, it follows that $f(\gamma(0))$ is a right nullvector of $X(1)-I_n$ and $f(\gamma(t))$ is a right nullvector of the linear operator
\begin{equation}
\label{eq:Gamma(t)def}
    \Gamma(t):=X(t)\left(X(1)-I_n\right)X^{-1}(t).
\end{equation}

Following \cite{Dankowicz2022329}, assume next that $\gamma$ is \textit{hyperbolic} in the sense that there exists a $1$-periodic, smooth family of projections of the form \begin{equation}
\label{eq:Q(t)perdef}
    Q(t):=X(t)(I_n-f(\gamma(0))w^\mathsf{T})X^{-1}(t)
\end{equation}
for some vector $w$, such that $\Gamma(t)$ is a bijection on the image of $Q(t)$. As shown below, this is equivalent to the standard definition of hyperbolicity in terms of the eigenvalues of the monodromy matrix $X(1)$. Importantly, the family $Q(t)$ and the complementary family $I_n-Q(t)$ respect the spectral decomposition of $\Gamma(t)+I_n$.
\subsubsection{The adjoint boundary-value problem}

The following sequence of lemmas and corollaries culminate in the formulation of the adjoint boundary-value problem in Lemma~\ref{lem:adjbvpper} and a projection of the local dynamics onto transversal directions relative to the limit cycle in Corollary~\ref{cor:transdynper}. The discussion is adapted from \cite{Dankowicz2022329} and organized to set up a point of departure for the subsequent analysis of noise-induced transversal dynamics. 

\begin{lemma}
\label{lem:normalizationper}
The vector $w$ satisfies the normalization condition
\begin{equation}
\label{eq:normalized}
w^\mathsf{T}f(\gamma(0))=1.
\end{equation}
\end{lemma}
\begin{proof}
Since $Q(0)$ is a projection, it is idempotent, i.e., $Q^2(0)=Q(0)$ or, by \eqref{eq:Q(t)perdef},
\begin{equation}
    I_n-f(\gamma(0))w^\mathsf{T}=\left(I_n-f(\gamma(0))w^\mathsf{T}\right)^2.
\end{equation}
Algebraic manipulation shows that this holds if and only if \eqref{eq:normalized} is satisfied.
\end{proof}

\begin{lemma}
\label{lem:imageandkernelper}
Let $\Ima Q(t)$ and $\ker Q(t)$ denote the image and kernel of the projection $Q(t)$. Then,
\begin{equation}
    \Ima Q(t)=\{X^{{-1},\mathsf{T}}(t)w\}^\perp\mbox{ and }\ker Q(t)=\mathrm{span}\{X(t)f(\gamma(0))\}.
\end{equation}
\end{lemma}
\begin{proof}
By \eqref{eq:Q(t)perdef} it suffices to show this for $t=0$. If $v\in\ker Q(0)$, then
\begin{equation}
    0=\left(I_n-f(\gamma(0))w^\mathsf{T}\right)v=v-f(\gamma(0))w^\mathsf{T}v,
\end{equation}
i.e., that $v\in\mathrm{span}\{f(\gamma(0))\}$. That $\ker Q(0)=\mathrm{span}\{f(\gamma(0))\}$ then follows from Lemma~\ref{lem:normalizationper}. Similarly, if $v\in\Ima Q(0)$, then there exists a vector $u$ such that
\begin{equation}
    v=(I_n-f(\gamma(0))w^\mathsf{T})u
\end{equation}
and, by Lemma~\ref{lem:normalizationper},
\begin{equation}
    w^\mathsf{T}v= \left(w^\mathsf{T}-w^\mathsf{T}f(\gamma(0))w^\mathsf{T}\right)u=0,
\end{equation}
i.e., $\Ima Q(0)\in\{w\}^\perp$. Equality then follows from the rank-nullity theorem.
\end{proof}

\begin{lemma}
\label{lem:leftnullvectorper}
The vector $\lambda(t):=X^{{-1},\mathsf{T}}(t)w$ is the unique left nullvector of $\Gamma(t)$ for which
\begin{equation}
\label{eq:normcondt}
    \lambda^\mathsf{T}(t)f(\gamma(t))=1.
\end{equation}
\end{lemma}
\begin{proof}
By \eqref{eq:Gamma(t)def} it suffices to show this for $t=0$, where \eqref{eq:normcondt} reduces to \eqref{eq:normalized}. To show that $w$ is a left nullvector of $\Gamma(0)=X(1)-I_n$, note that the left nullvector property $w^\mathsf{T}\Gamma(0)=0$ holds if and only if $w^\mathsf{T}X(1)=w^\mathsf{T}$. But the latter follows by $1$-periodicity of $Q$, since $Q(0)=Q(1)$ implies that
\begin{equation}
    X(1)(I_n-f(\gamma(0))w^\mathsf{T})X^{-1}(1)=I_n-f(\gamma(0))w^\mathsf{T},
\end{equation}
which reduces to  $f(\gamma(0))w^\top=f(\gamma(0))w^\top X(1)$ by direct algebraic manipulation and use of the fact that $X(1)f(\gamma(0))=f(\gamma(0))$. Finally, suppose that $\tilde{w}$ is any left nullvector of $\Gamma(0)$ and let $v$ denote an arbitrary vector in $\Ima Q(0)$. Since $\Gamma(0)$ is onto $\Ima Q(0)$, there exists a $u$ such that $\Gamma(0)u=v$ and, consequently, $\tilde{w}^\mathsf{T}v=\tilde{w}^\mathsf{T}\Gamma(0)u=0$. By Lemma~\ref{lem:imageandkernelper}, $\tilde{w}\in\mathrm{span}\{w\}$ and the uniqueness claim follows as well.
\end{proof}

\begin{corollary}
The right nullspace of $\Gamma(t)$ consists of vectors tangential to the curve $s\mapsto\gamma(s)$ at $s=t$.
\end{corollary}
\begin{proof}
By \eqref{eq:Gamma(t)def} it suffices to show this for $t=0$. By substitution from \eqref{eq:Gamma(t)def} and \eqref{eq:Q(t)perdef}, it follows that
\begin{equation}
\label{eq:G0Q0=G0}
    \Gamma(0)Q(0)=\left(X(1)-I_n\right)\left(I_n-f(\gamma(0))w^\mathsf{T}\right)=\Gamma(0),
\end{equation}
where we again used the fact that $X(1)f(\gamma(0))=f(\gamma(0))$. Now suppose that $v$ is a right nullvector of $\Gamma(0)$. It follows that $\Gamma(0)Q(0)v=\Gamma(0)v=0$. Since $\Gamma(0)$ is one-to-one on the image of $Q(0)$, $Q(0)v=0$ and, by Lemma~\ref{lem:imageandkernelper}, $v\in\mathrm{span}\{f(\gamma(0))\}$.
\end{proof}

Conversely, suppose that the eigenspace of the monodromy matrix $X(1)$ corresponding to the eigenvalue $1$ is given by $\mathrm{span}\{f(\gamma(0))\}$ and choose $w$ in the construction of $Q(t)$ as the unique left nullvector of $\Gamma(0)$ such that \eqref{eq:normalized} holds. Then, by \eqref{eq:G0Q0=G0}, $v\in\Ima Q(0)$ and $\Gamma(0)v=0$ imply that $v=0$. Similarly, since
\begin{equation}
    Q(0)\Gamma(0)=\left(I_n-f(\gamma(0))w^\mathsf{T}\right)\left(X(1)-I_n\right)=\Gamma(0),
\end{equation}
it follows from the rank-nullity theorem that if $v\in\Ima Q(0)$ then there exists a $\tilde{v}\in\Ima Q(0)$ such that $\Gamma(0)\tilde{v}=v$. We conclude that $\Gamma(0)$ is a bijection on the image of $Q(0)$ and, consequently, that $\Gamma(t)$ is a bijection on the image of $Q(t)$.

\begin{lemma}
\label{lem:adjbvpper}
The function $\lambda(t):=X^{{-1},\mathsf{T}}(t)w$ is the unique solution to the adjoint boundary-value problem
\begin{equation}
\label{eq:adjtbvpper}
0=-\dot{\mu}^\mathsf{T}-T\mu^\mathsf{T}\mathrm{D} f(\gamma),\quad 0=\mu^\mathsf{T}(0)-\mu^\mathsf{T}(1),\quad 1=\int_0^1\mu^\mathsf{T}f(\gamma(s))\,\mathrm{d}t.
\end{equation}
\end{lemma}

\begin{proof}
Let $\mu$ denote an arbitrary solution to this boundary-value problem. It follows that $\mu^\mathsf{T}(t)X(t)\equiv w^\mathsf{T}$ and $\mu^\mathsf{T}(t)f(\gamma(s))\equiv 1$, since i) differentiation and use of \eqref{eq:varprobper} and \eqref{eq:adjtbvpper} shows that the left-hand sides are constant, ii) the integral condition implies that the second constant equals $1$ and, consequently, $\mu^\mathsf{T}(0)f(\gamma(0))=1$, iii) periodicity shows that $\mu^\mathsf{T}(0)X(1)=\mu^\mathsf{T}(0)$ and, by the proof of Lemma~\ref{lem:leftnullvectorper}, that $\mu^\mathsf{T}(0)$ is a left nullvector of $\Gamma(0)$ that iv) by Lemma~\ref{lem:leftnullvectorper} must equal $w$.
\end{proof}

\begin{corollary}
In the notation of the previous lemma,
\begin{equation}
Q(t)=I_n-f(\gamma(s))\lambda^\mathsf{T}(t)
\end{equation}
and projects along $f(\gamma(s))$ onto a co-dimension-one hyperplane orthogonal to $\lambda(t)$.
\end{corollary}

\begin{lemma}
The family of operators $P(t)=Q(t)$ is the unique periodic family of projections onto co-dimension-one hyperplanes which satisfies the invariance condition
\begin{equation}
\label{eq:QXinvariance}
X(t)P(0)=P(t)X(t)
\end{equation}
and such that $\Gamma(t)$ is a bijection on the image of $P(t)$.
\end{lemma}

\begin{proof}
For any such family of projections, there exist two periodic functions $\phi$ and $\mu$, such that $P(t)=I_n-\phi(t)\mu^\mathsf{T}(t)$, $\mu^\mathsf{T}(t)\phi(t)\equiv 1$ and, by invariance, $\phi(t)=X(t)\phi(0)$ and $\mu(t)=X^{-1,\mathsf{T}}(t)\mu(0)$. By direct algebraic manipulation and periodicity, we obtain
\begin{equation}
    \Gamma(0)P(0)=\Gamma(0)=P(0)\Gamma(0).
\end{equation}
It follows from the first equality that $P(0)f(\gamma(0))$ is a nullvector of $\Gamma(0)$ and, by injectivity of $\Gamma(0)$ on $\Ima P(0)$, that  $P(0)f(\gamma(0))=0$. Without loss of generality, we conclude that $\phi(0)=f(\gamma(0))$. From the second equality, we obtain $\mu^T(0)\Gamma(0)=0$. By surjectivity of $\Gamma(0)$ on $\Ima P(0)$, as in the proof of Lemma~\ref{lem:leftnullvectorper}, then $\mu(0)=w$.
\end{proof}

\begin{corollary}
The stable and unstable eigenspaces of $\Gamma(t)+I_n$ lie in the image of $Q(t)$.
\end{corollary}

We proceed to investigate the local geometry on a neighborhood of $\gamma$ characterized in terms of the projections $Q$. These provide a local decomposition into state-space displacements tangential and transversal to the periodic orbit. When convenient, we perform formal calculations with differentials in order to capture leading-order effects.

\begin{lemma}
\label{eq:decompper}
For $x$ sufficiently close to $\gamma(t)$ for some $t$, there exists a unique time $\tau$ (with $\tau\approx t$) such that
\begin{equation}
x=\gamma(\tau)+x_\mathrm{tr},
\end{equation}
where $\lambda(\tau)^\mathsf{T}x_\mathrm{tr}=0$ and, consequently, $Q(\tau)x_\mathrm{tr}=x_\mathrm{tr}$.
\end{lemma}
\begin{proof}
Let $H(x,\tau)=\lambda^\mathsf{T}(\tau)\left(x-\gamma(\tau)\right)$. It follows that $H(\gamma(t),t)=0$ and
\begin{equation}
\partial_\tau H(\gamma(t),t)=-T\lambda^\mathsf{T}(t)f(\gamma(s))=-T.
\end{equation}
The claim follows from the implicit function theorem.
\end{proof}

\begin{corollary}
\label{cor:transdynper}
For the differential $\mathrm{d}x$, the result of Lemma~\ref{eq:decompper} implies the existence of differentials $\mathrm{d}\tau$ and $\mathrm{d}x_\mathrm{tr}$, such that 
\begin{equation}
\label{eq:x+dx}
x+\mathrm{d}x=\gamma(\tau+\mathrm{d}\tau)+x_\mathrm{tr}+\mathrm{d}x_\mathrm{tr},
\end{equation}
where $\lambda^\mathsf{T}(\tau+\mathrm{d}\tau)(x_\mathrm{tr}+\mathrm{d}x_\mathrm{tr})=0$ and, consequently,
\begin{equation}
\label{eq:Q(tau+dtau)}
Q(\tau+\mathrm{d}\tau)(x_\mathrm{tr}+\mathrm{d}x_\mathrm{tr})=x_\mathrm{tr}+\mathrm{d}x_\mathrm{tr}.
\end{equation}
It follows that
\begin{equation}
\label{eq:dx_tr in Q}
\mathrm{d}x_\mathrm{tr}-T\mathrm{D} f(\gamma(\tau))x_\mathrm{tr}\mathrm{d}\tau=Q(\tau)\left(\mathrm{d}x-T\mathrm{D} f(\gamma(\tau))x_\mathrm{tr}\mathrm{d}\tau\right).
\end{equation}
\end{corollary}

\begin{proof}
Eq.~\eqref{eq:x+dx} implies that
\begin{equation}
\label{eq:dx_tr1}
\mathrm{d}x_\mathrm{tr}=\mathrm{d}x-Tf(\gamma(\tau))\mathrm{d}\tau.
\end{equation}
Similarly, from \eqref{eq:Q(tau+dtau)}, we obtain
\begin{equation}
\label{eq:dx_tr2}
\mathrm{d}x_\mathrm{tr}=\dot{Q}(\tau)x_\mathrm{tr}\mathrm{d}\tau+Q(\tau)\mathrm{d}x_\mathrm{tr}.
\end{equation}
From the definition of $Q$ it follows that
\begin{equation}
\dot{Q}(\tau)=-T\mathrm{D} f(\gamma(\tau)) f(\gamma(\tau))\lambda^\mathsf{T}(\tau)+Tf(\gamma(\tau))\lambda^\mathsf{T}(\tau)\mathrm{D} f(\gamma(\tau))
\end{equation}
and, consequently, that
\begin{equation}
\label{eq:Qdotx_tr}
\dot{Q}(\tau)x_\mathrm{tr}=T(I_n-Q(\tau))\mathrm{D} f(\gamma(\tau))x_\mathrm{tr},
\end{equation}
since $\lambda^\mathsf{T}(\tau)x_\mathrm{tr}=0$. Substitution of \eqref{eq:dx_tr1} and \eqref{eq:Qdotx_tr} into the right-hand side of \eqref{eq:dx_tr2} yields \eqref{eq:dx_tr in Q}, since $Q(\tau)f(\gamma(\tau))=0$.
\end{proof}

\subsubsection{The covariance boundary-value problem}
In lieu of the deterministic dynamics in \eqref{eq:determ}, now consider the It\^{o} SDE
\begin{equation}
\label{eq:sde_first}
\mathrm{d}x=Tf(x)\mathrm{d}t+\sigma \sqrt{T}F(x)\mathrm{d}W_{t}
\end{equation} 
in terms of the noise intensity $\sigma$ and a vector $W_{t}\in\mathbb{R}^m$ of independent standard Brownian motions. We consider the limit of small noise intensity and are concerned only with leading-order expansions in $\sigma$. In the notation of the preceding lemmas, we continue to calculate with differentials. Since we are only interested in first-order terms in $\sigma$, we drop any It\^o corrections arising from the It\^o-formula. The following proposition is then relatively straightforward.

\begin{proposition}
For $x$ sufficiently close to $\gamma(t)$ for some $t$, to first order in $\|x_\mathrm{tr}\|$ and $\sigma$,
\begin{equation}
\label{eq:OUprocess}
\mathrm{d}x_\mathrm{tr}=T\mathrm{D} f(\gamma(\tau))x_\mathrm{tr}\mathrm{d}\tau+\sigma \sqrt{T}Q(\tau)F(\gamma(\tau))\mathrm{d}W_{\tau}.
\end{equation}
\end{proposition}

\begin{proof}
For $x_\mathrm{tr}=0$ and $\sigma=0$, we have $\tau=t$ and $\mathrm{d}x_\mathrm{tr}=0$. In this limit, \eqref{eq:sde_first} and \eqref{eq:dx_tr1} imply that
\begin{equation}
0=Tf\left(\gamma(\tau)\right)(\mathrm{d}t-\mathrm{d}\tau),
\end{equation}
from which it follows that $\mathrm{d}\tau=\mathrm{d}t$ to zeroth order in $\|x_\mathrm{tr}\|$ and $\sigma$. To first order in $\|x_\mathrm{tr}\|$ and $\sigma$, \eqref{eq:sde_first} then yields
\begin{equation}
Q(\tau)\left(\mathrm{d}x-T\mathrm{D}f(\gamma(\tau))x_\mathrm{tr}\mathrm{d}\tau\right)=\sigma\sqrt{T}Q(\tau) F(\gamma(\tau))\mathrm{d}W_\tau
\end{equation}
and substitution in \eqref{eq:dx_tr in Q} yields the claimed result.
\end{proof}

The Ornstein-Uhlenbeck process in \eqref{eq:OUprocess} is a leading-order approximation of the local dynamics near $\gamma$ that neglects the cumulative effects of higher-order terms in $\|x_\mathrm{tr}\|$ and $\sigma$. Given an initial condition $x_\mathrm{tr}(0)$, \eqref{eq:OUprocess} may be solved explicitly to yield
\begin{align}
x_\mathrm{tr}(\tau)&=X(\tau)x_\mathrm{tr}(0)+\sigma\sqrt{T}X(\tau)\int_{0}^{\tau}X^{-1}(s)Q(s)F(\gamma(s)))\mathrm{d}W_{s}\nonumber\\
    &=X(\tau)x_\mathrm{tr}(0)+\sigma\sqrt{T}Q(\tau)X(\tau)\int_{0}^{\tau}G(s)\mathrm{d}W_{s},
\end{align}
where $G(s)=X^{-1}(s)F(\gamma(s))$ and the invariance condition \eqref{eq:QXinvariance} with $P(t)=Q(t)$ was used twice. It follows that $\lambda^\mathsf{T}x_\mathrm{tr}\equiv 0$ provided that $w^\mathsf{T}x_\mathrm{tr}(0)=0$. The rescaled covariance matrix
\begin{equation}
\label{eq:cov_def}
    C:=\frac{1}{\sigma^2}\mathbb{E}\left[ x_\mathrm{tr}x_\mathrm{tr}^{\mathrm{T}}\right]
\end{equation}
is then given by
\begin{align}
\label{eq:Ctr_integral}
C(\tau)&=X(\tau)C(0)X^\mathsf{T}(\tau)+ TQ(\tau)X(\tau)\left(\int_{0}^{\tau}G(s)G^\mathsf{T}(s)\mathrm{d}s\right)X^\mathsf{T}(\tau)Q^\mathsf{T}(\tau),
\end{align}
where we used the stochastic equalities~\cite{higham2021introduction}
\begin{equation}
\label{stoch_equality_1}
    \mathbb{E}\left[\int_0^\tau Y(s)\mathrm{d}W_s\right]=0
\end{equation}
and
\begin{equation}
\label{stoch_equality_2}
    \mathbb{E}\left[\int_0^\tau Y(s)\mathrm{d}W_s\int_0^\tau Z^\mathsf{T}(s)\mathrm{d}W_s\right]=\int_0^\tau Y(s)Z^\mathsf{T}(s)\mathrm{d}s.
\end{equation}
It follows that $\lambda^\mathsf{T}C\lambda\equiv w^\mathsf{T}C(0)w$ and that $\lambda^\mathsf{T}C\lambda\equiv0$ if and only if $w^\mathsf{T}C(0)w=0$.

\begin{lemma}
Let 
\begin{equation}
\mathcal{I}=\int_0^1G(s)G^\mathsf{T}(s)\mathrm{d}s.
\end{equation}
Then, for arbitrary $C(0)$ and $k\in\mathbb{N}\cup\{0\}$,
\begin{equation}
C(k+1)=X(1)C(k)X^\mathsf{T}(1)+TQ(0)X(1)\mathcal{I}X^\mathsf{T}(1)Q^\mathsf{T}(0).
\end{equation}
\end{lemma}

\begin{proof}
By \eqref{eq:Ctr_integral},
\begin{equation}
C(k)=X(k)C(0)X^\mathsf{T}(k)+TQ(0)X(k)\left(\int_0^{k}G(s)G^\mathsf{T}(s)\mathrm{d}s\right)X^\mathsf{T}(k)Q^\mathsf{T}(0).
\end{equation}
Moreover, by the definition of $G$, we have
\begin{equation}
\mathcal{I}=X(k)\left(\int_k^{k+1}G(s)G^\mathsf{T}(s)\mathrm{d}s\right)X^\mathsf{T}(k).
\end{equation}
The claim follows by substitution.
\end{proof}

We proceed to assume that the unstable eigenspace of $X(1)$ is empty, such that
\begin{equation}
X(k)=f(\gamma(0))w^\mathsf{T}+\mathcal{O}\left(\mathrm{exp}(-k/\tau_{\mathrm{tr}})\right)
\end{equation}
for some positive constant $\tau_\mathrm{tr}$ and large $k$ (here, the asymptotics is with respect to large times so that the $\mathcal{O}(\cdot)$-term becomes negligible as k becomes large). 

\begin{proposition}
\label{prop: perorbseries}
Given a non-negative constant $b$, there exists an initial condition $C(0)$ with $w^\mathsf{T}C(0)w=b$, such that $C(k)=C(0)$ for all $k$.
\end{proposition}
\begin{proof}
The previous lemma shows that $C(1)=C(0)\Rightarrow C(k)=C(0)$ for all $k$. Suppose that
\begin{align}
\label{eq:C0explicit}
C(0)&=b f(\gamma(0))f^\mathsf{T}(\gamma(0))+T\sum_{n=1}^\infty Q(0)X(n)\mathcal{I}X^\mathsf{T}(n)Q^\mathsf{T}(0).
\end{align}
The series converges by the assumption on $X(1)$. By substitution and use of periodicity and the invariance condition
\begin{equation}
X(1)C(0)X^\mathsf{T}(1)=T\sum_{n=2}^\infty Q(0)X(n)\mathcal{I}X^\mathsf{T}(n)Q^\mathsf{T}(0)
\end{equation}
and the claim follows by evaluation of \eqref{eq:Ctr_integral} at $\tau=1$.
\end{proof}
\begin{corollary}
The function $C(\tau)$ in \eqref{eq:Ctr_integral} with $w^\mathsf{T}C(0)w=b$ converges to the periodic function obtained with $C(0)$ given by \eqref{eq:C0explicit} as $\tau\rightarrow\infty$.
\end{corollary}

\begin{proposition}
\label{prop: perorbbvp}
Let $C_\mathrm{per}$ denote the function in \eqref{eq:Ctr_integral} obtained with $C(0)$ given by \eqref{eq:C0explicit} when $b=0$. It follows that, when $a=0$, $C_\mathrm{per}$ is the unique periodic solution of the periodic Lyapunov equation (cf.~\cite{bolzern1988periodic})
\begin{align}
\label{eq:covardiff}
\dot{C}&=T\left(\mathrm{D} f(\gamma)C+C\mathrm{D} f^\mathsf{T}(\gamma)+QF(\gamma)F^\mathsf{T}(\gamma)Q^\mathsf{T}+a f(\gamma)f^\mathsf{T}(\gamma)\right),
\end{align}
for which $w^\mathsf{T}C(0)w=0$ and that no periodic solution exists otherwise.
\end{proposition}
\begin{proof}
It follows from \eqref{eq:covardiff} that
\begin{equation}
\frac{d}{dt}\lambda^\mathsf{T}C\lambda=a T
\end{equation}
Thus, $C$ is periodic only if $a=0$. The transformation $C=X\tilde{C}X^\mathsf{T}$ then implies that
\begin{equation}
\dot{\tilde{C}}=TQ(0)GG^\mathsf{T}Q^\mathsf{T}(0)
\end{equation}
and the conclusion follows by integration.
\end{proof}

\begin{corollary}
\label{cor:perorbinvform}
Let $\mathfrak{vec}$ denote the vectorization operator that turns a matrix into a column vector. Then, $\mathfrak{vec}\left(C_\mathrm{per}(0)\right)$ may be obtained from the first $n^2$ components of the image of the column vector
\begin{align}
    \begin{pmatrix}\mathfrak{vec}\left((Q(0)X(1)\mathcal{I}X^\mathsf{T}(1)Q^\mathsf{T}(0)\right)\\0\end{pmatrix}
\end{align}
under multiplication by the matrix
\begin{equation}
    \begin{pmatrix}I_{2n}-X(1)\otimes X(1) & f(\gamma(0))\otimes f(\gamma(0))\\w^\mathsf{T}\otimes w^\mathsf{T} & 0\end{pmatrix}^{-1}\nonumber\\
\end{equation}
where the inverse exists since no pair of eigenvalues of $X(1)$ are reciprocal numbers.
\end{corollary}

\subsection{Numerical examples}
\label{sec:numexpper}
An analytically tractable example (a version of which was considered in \cite{10.1115/DETC2022-91153}) is given by the periodic solution
\begin{equation}
\label{eq:hopf_init_sol}
\gamma(t)=\begin{pmatrix}\cos 2\pi t\\\sin 2\pi t\end{pmatrix}
\end{equation}
with period $1$ of the dynamical system \eqref{eq:determ} with
\begin{equation}
\label{eq: hopf normal form}
f(x)=\begin{pmatrix}\vspace{1mm}x_1-x_2-x_1(x_1^2+x_2^2)\\x_1+x_2-x_2(x_1^2+x_2^2)\end{pmatrix}
\end{equation}
and $T=2\pi$, such that
\begin{equation}
f(\gamma(t))=\begin{pmatrix}-\sin 2\pi t\\\cos 2\pi t\end{pmatrix}.
\end{equation}
Since
\begin{equation}
X(t)=\begin{pmatrix}e^{-4\pi t}\cos 2\pi t & -\sin 2\pi t \\e^{-4\pi t}\sin 2\pi t & \cos 2\pi t \end{pmatrix},
\end{equation}
the eigenvalues of the monodromy matrix $X(1)$ equal $1$ and $e^{-4\pi}$ with eigenspaces spanned by $f(\gamma(0))=(0,1)^\mathsf{T}$ and $(1,0)^\mathsf{T}$, respectively. Then, $w^\mathsf{T}=(0,1)$ and, by definition,
\begin{equation}
\lambda(t)=X^{-1,\mathsf{T}}(t)w=\begin{pmatrix}-\sin 2\pi t\\\cos 2\pi t\end{pmatrix}.
\end{equation}
and
\begin{equation}
Q(t)=I_2-f(\gamma(t))\lambda^\mathsf{T}(t)=\begin{pmatrix}\cos^22\pi t & \cos 2\pi t\sin 2\pi t\\\cos 2\pi t\sin 2\pi t & \sin^22\pi t\end{pmatrix}.
\end{equation} 
It follows that
\begin{equation}
TQ(0)\left(\int_0^t G(s)G^\mathsf{T}(s)\mathrm{d}s\right)Q^\mathsf{T}(0)=\begin{pmatrix}\mathcal{I}(t) & 0\\0 & 0\end{pmatrix},
\end{equation}
where $\mathcal{I}(0)=0$. The result of Proposition~\ref{prop: perorbseries} then implies that every periodic solution of \eqref{eq:Ctr_integral} is of the form
\begin{align}
C(t)&=e^{-8\pi t}\left(\mathcal{I}(t)+\frac{\mathcal{I}(1)}{e^{8\pi}-1}\right)\begin{pmatrix}\cos^22\pi t & \cos 2\pi t\sin 2\pi t\\\cos 2\pi t\sin 2\pi t & \sin^22\pi t\end{pmatrix}\nonumber\\
&\quad+b\begin{pmatrix}\sin^2 2\pi t & -\sin 2\pi t\cos 2\pi t\\-\sin2\pi t\cos2\pi t & \cos^2 2\pi t\end{pmatrix}
\end{align}
for some $b$, such that $\lambda^\mathsf{T}C\lambda\equiv b$. We obtain a periodic rescaled covariance function provided that $b=0$. Its eigenvalues equal
\begin{equation}
e^{-8\pi t}\left(\mathcal{I}(t)+\frac{\mathcal{I}(1)}{e^{8\pi}-1}\right)
\end{equation}
and $0$ with eigenvectors $(\cos 2\pi t,\sin 2\pi t)^\mathsf{T}$ and $(-\sin 2\pi t, \cos 2\pi t)^\mathsf{T}$, respectively. With the noise defined in terms of the matrix
\begin{equation}
F(x)=\begin{pmatrix}x_1x_2\\x_2^2 \end{pmatrix}
\end{equation}
and a single standard Brownian motion, we obtain
\begin{equation}
\mathcal{I}(t)=\frac{e^{8\pi t}(5-4\cos 4\pi t-2\sin 4\pi t)-1}{40}
\end{equation}
and the nonzero eigenvalue of the periodic rescaled covariance function equals
\begin{equation}
\label{eq:eigenvalue}
\frac{5-4\cos 4\pi t-2\sin 4 \pi t}{40}.
\end{equation}
 
We compare the analytical predictions with the approximate results obtained from an implementation of the periodic boundary-value problem in the software package \textsc{coco} using continuous, piecewise-polynomial interpolants for $\gamma(t)$, $\lambda(t)$, and $C(t)$ with base points on a uniform mesh and constrained by a discretization of the governing differential equations at collocation nodes associated with each mesh interval. Figure~\ref{fig1: po_ex_1} shows the non-trivial eigenvalue of the covariance matrix obtained using~\textsc{coco} and the analytical expression~\eqref{eq:eigenvalue}. The results are in excellent agreement.

\begin{figure}[ht]
    \centering
    {\includegraphics[width=0.65\textwidth]{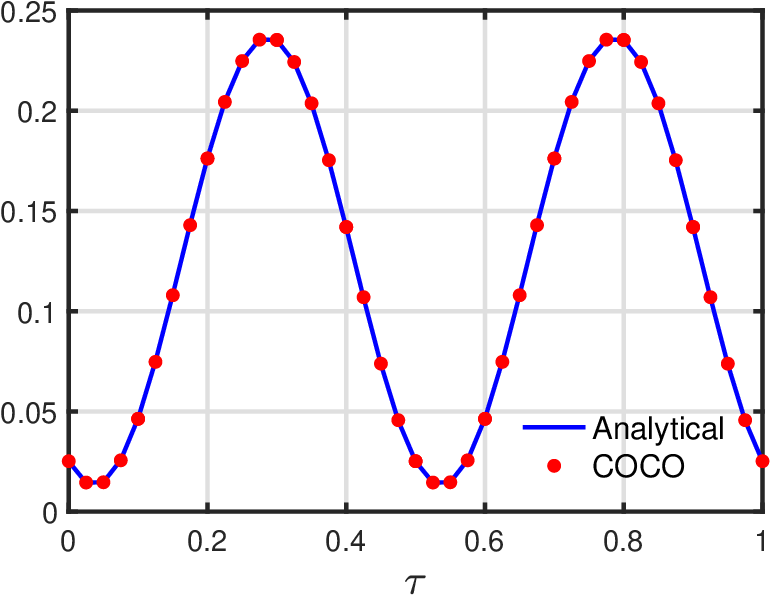}}
    \caption{Comparison of the non-trivial eigenvalue of the covariance matrix function for the vector field in Eq.~\eqref{eq: hopf normal form} obtained from a discretization of the governing boundary-value problem in Proposition~\ref{prop: perorbbvp} in the software package \textsc{coco} and from the analytical expression reported in~\eqref{eq:eigenvalue}, respectively.}
    \label{fig1: po_ex_1}
\end{figure}

Finally, we validate our leading-order analysis by comparing the predictions with direct numerical simulation of the original It\^{o} SDE. To generate the stochastic trajectories, we use an explicit Euler Maruyama integrator in~\textsc{Matlab} with time-step $\mathrm{d}t=10^{-4}$, noise intensity $\sigma=0.1$, and Brownian increments generated using the \texttt{randn} command. The system is integrated for a duration of 500 periods of the underlying limit cycle. In Fig.~\ref{fig2: po_ex_1}, we show the resultant stochastic trajectories projected along the normalized radial eigenvector $e(\tau)=(\cos 2\pi \tau,\sin 2\pi \tau)^\mathsf{T}$ corresponding to the nonzero eigenvalue of the covariance matrix $C(\tau)$. Specifically, for each point on the stochastic trajectory $x(t)$, we use the hyperplane definition to compute the corresponding value of $\tau(t)$ and project the perturbation $x(t)-\gamma(\tau(t))$ onto $e(\tau(t))$. These results are compared with the one and two standard deviation predictions computed from $\sigma$ times the square root of the eigenvalue given in \eqref{eq:eigenvalue}. We observe good qualitative agreement between the predicted variations and those obtained from direct integration.

\begin{figure}[htp]
    \centering
    {\includegraphics[width=0.55\textwidth]{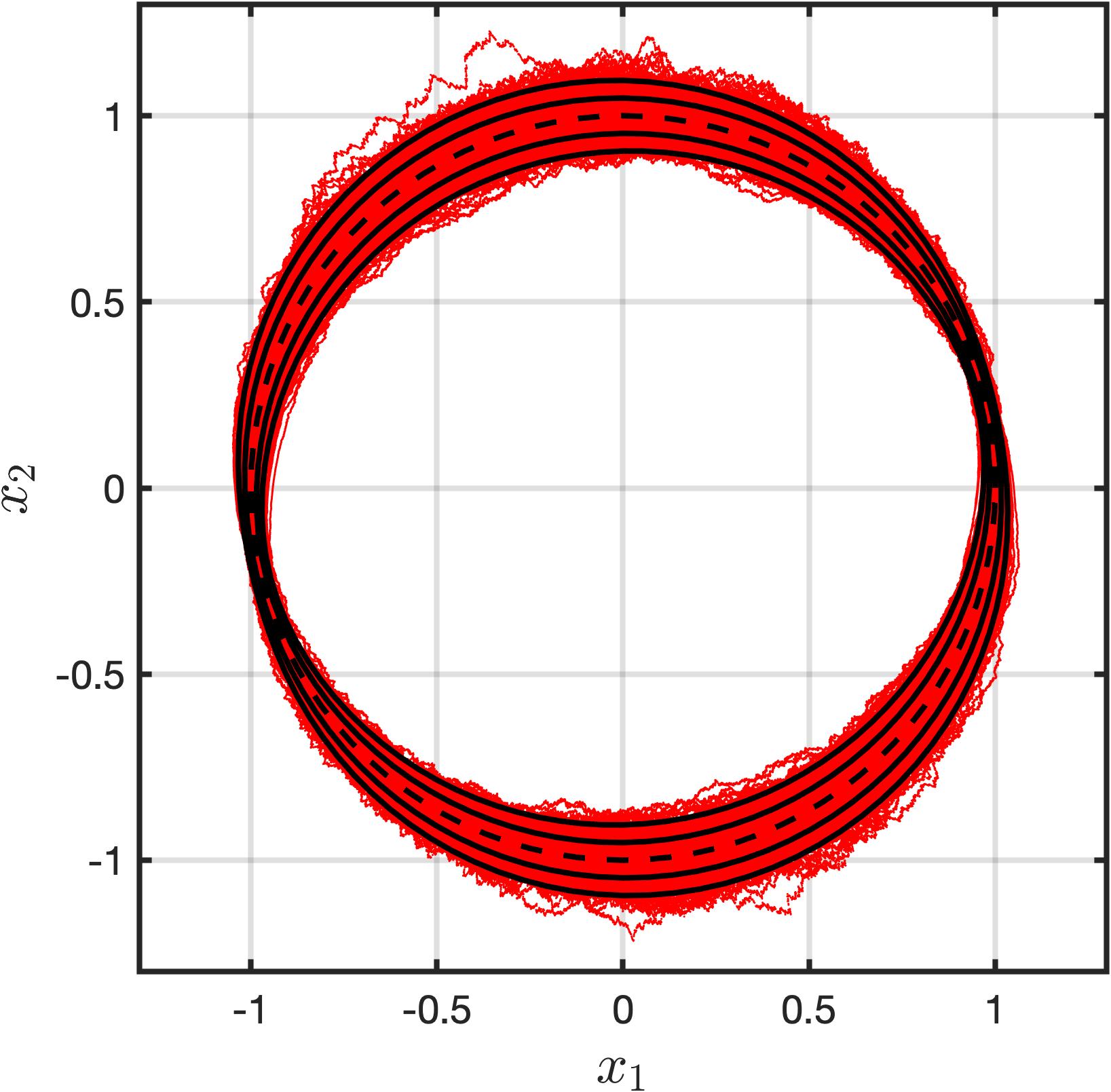}}
    \bigskip
    
    {\includegraphics[width=0.65\textwidth]{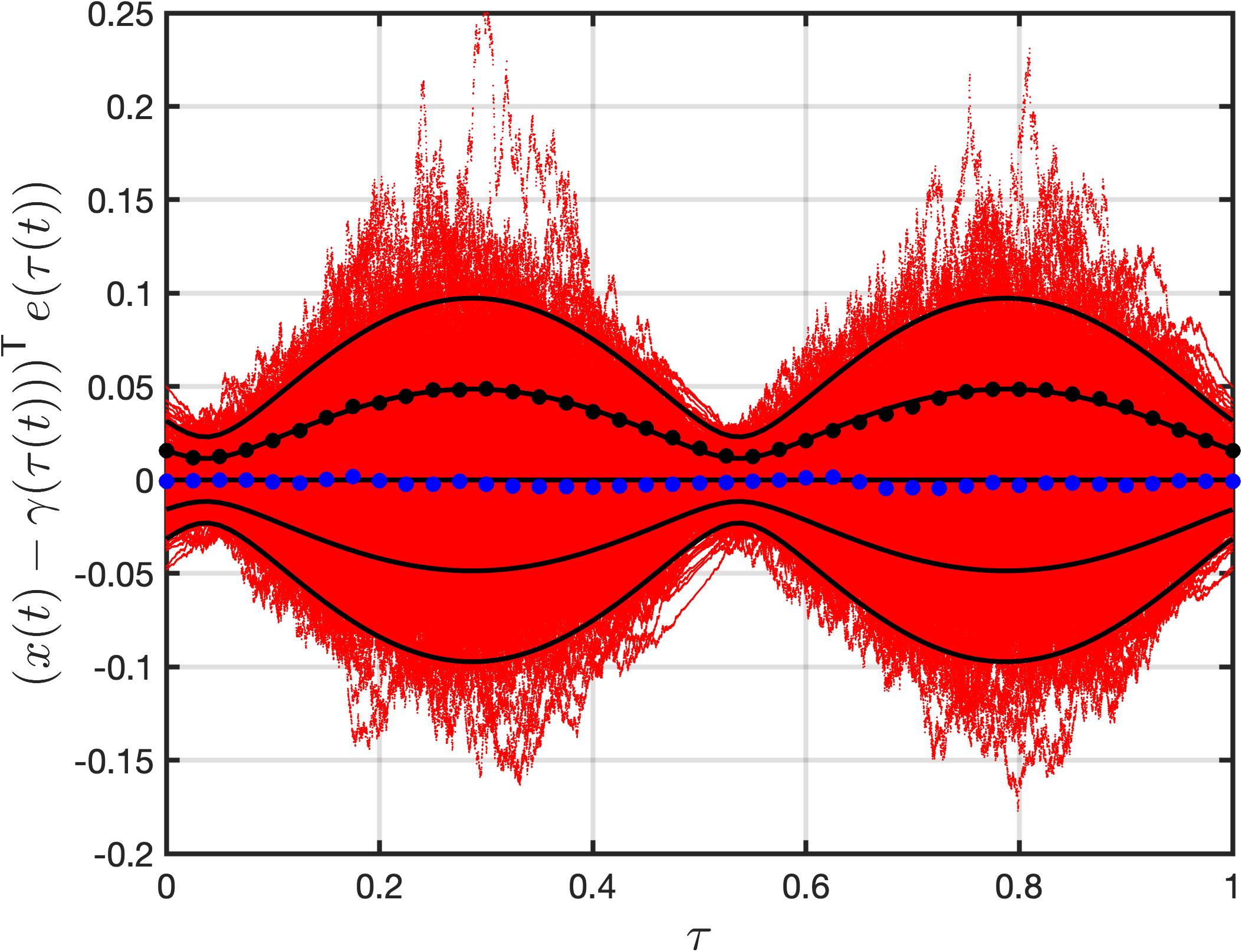}}
    \caption{Comparison of leading-order theoretical predictions and a numerical time history describing the influence of small Brownian noise on the dynamics near a limit cycle of the deterministic vector field Eq.~\eqref{eq: hopf normal form} in cartesian (upper panel) and polar (lower panel) coordinates. Stochastic trajectories (red dots) were obtained using an Euler-Maruyama scheme with $\sigma=0.1$ and $\mathrm{d}t=10^{-4}$ for $500$ periods of the limit cycle. Dashed curves represent the deterministic limit cycle while solid curves represent predicted deviations from the limit cycle equal to one and two standard deviations, respectively, computed using \eqref{eq:eigenvalue}. In the lower panel, the vertical axis equals the radial deviation from the limit cycle, since  the normalized radial eigenvector $e(\tau(t))=\gamma(\tau(t))$ and $\lambda(\tau(t))^\mathsf{T}\left(x(t)-\gamma(\tau(t))\right)=0$. There, filled circles represent the statistical mean (blue) and standard deviation (black) of simulated data collected in bins of width $\Delta\tau=1/40$.}
    \label{fig2: po_ex_1}
\end{figure}

In the case of periodically excited deterministic dynamics, the analysis in the previous section carries over in its entirety by assigning the identity matrix to $Q$ in \eqref{eq:C0explicit} and \eqref{eq:Ctr_integral}. Consider, for example, the periodic solution
\begin{equation}
    \gamma(t)=\begin{pmatrix}\sin2\pi t\\\cos2\pi t\end{pmatrix}
\end{equation}
with period $1$ of the linear $1$-periodic dynamical system
\begin{equation}
    \dot{x}=2\pi f(t,x),\quad\mbox{ where }f(t,x):=\begin{pmatrix}x_2\\-2x_2-x_1+2\cos 2\pi t\end{pmatrix},
\end{equation}
such that $T=2\pi$ and
\begin{equation}
    f(t,\gamma(t))=\begin{pmatrix}\cos2\pi t\\-\sin2\pi t\end{pmatrix}.
\end{equation}
In this case
\begin{equation}
    X(t)=e^{-2\pi t}\begin{pmatrix}1+2\pi t & 2\pi t\\-2\pi t& 1-2\pi t\end{pmatrix}
\end{equation}
with the (sole) eigenvalue of the monodromy matrix $X(1)$ equal to $e^{-2\pi}$. Since the vector field is non-autonomous, it is not the case that $X^{-1}(t)f(t,\gamma(t))\equiv f(0,\gamma(0))$, but this property is also no longer needed for the construction of $Q$.

Now suppose that $F(x)=(0,x_1)^\mathsf{T}$. Substitution in \eqref{eq:C0explicit} then yields
\begin{equation}
    C(0)=\frac{1}{32}\begin{pmatrix}5 & -1 \\-1 & 1\end{pmatrix}
\end{equation}
and from \eqref{eq:Ctr_integral},
\begin{equation}
\label{eq:ex2pred}
    C(t)=\frac{1}{32}\begin{pmatrix}4+\cos4\pi t-\sin4\pi t & -\cos4\pi t-\sin4\pi t\\-\cos4\pi t-\sin4\pi t & 4-3\cos4\pi t-\sin4\pi t\end{pmatrix}
\end{equation}
with pair of positive eigenvalues
\begin{equation}
\label{eq:eigslinper}
    \frac{4-\cos4\pi t-\sin4\pi t\pm\sqrt{3+2\cos8\pi t+\sin8\pi t}}{32}.
\end{equation}
The panels in Fig.~\ref{fig3: po_ex_2} demonstrate excellent agreement between the density clouds of points of intersection of stochastic trajectories with several constant-phase sections and ellipses spanned by the eigenvectors of the corresponding covariance matrix and with radii given by the square roots of the eigenvalues in \eqref{eq:eigslinper}.

\begin{figure}[htbp]
    \centering
    \includegraphics[width=0.65\textwidth]{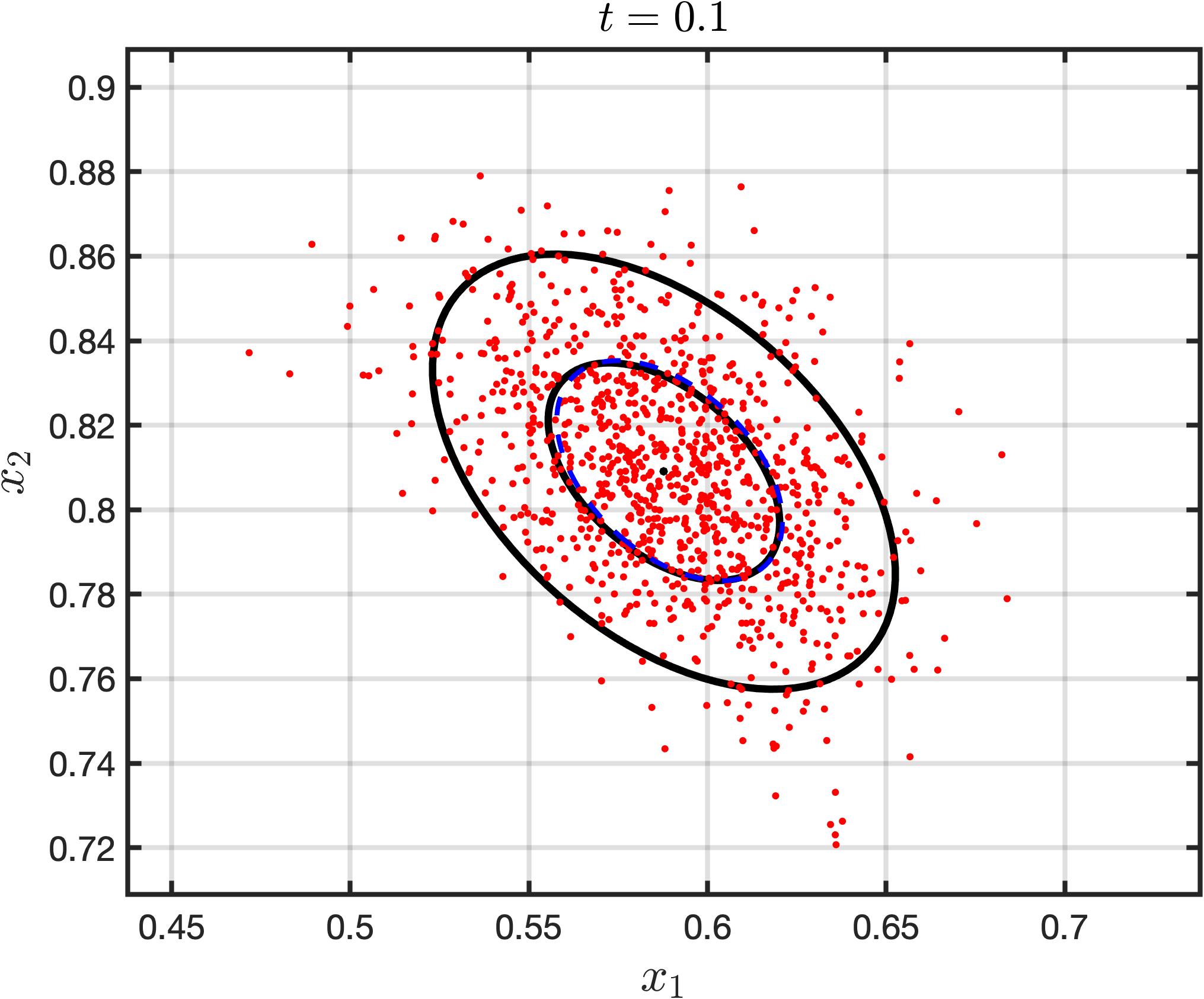}
    \bigskip
    
    \includegraphics[width=0.65\textwidth]{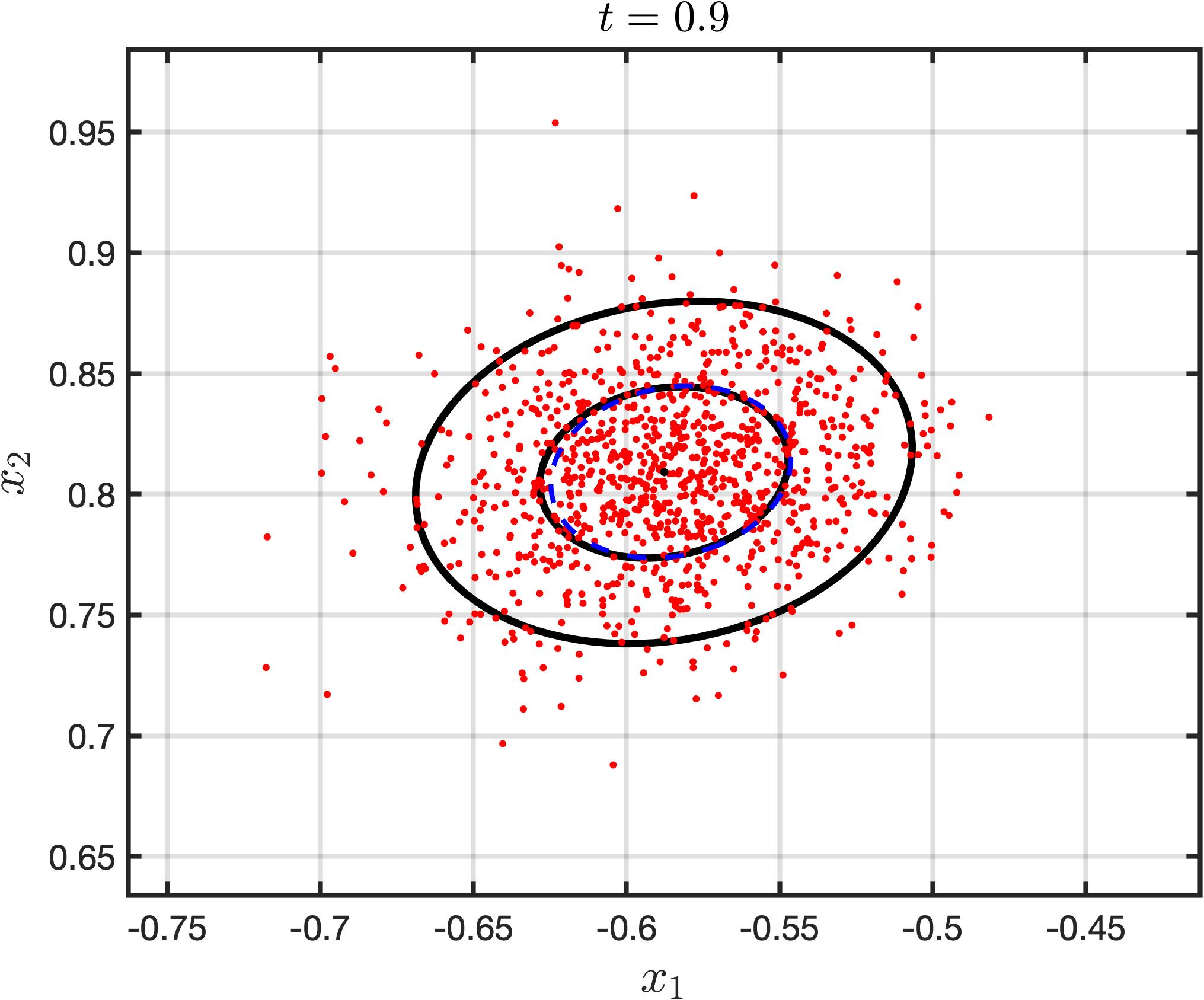}
    \caption{Comparison of leading-order theoretical predictions and a numerical time history describing the influence of small Brownian noise on the dynamics near a limit cycle of the harmonically excited, damped, linear oscillator in Sec.~\ref{sec:numexpper}. The top and bottom panels show samples (red dots) in sections of constant excitation phase ($t=0.1$ and $t=0.9$, respectively) obtained using an Euler-Maruyama scheme with $\sigma=0.1$ and $\mathrm{d}t=10^{-4}$ for $1,000$ periods of excitation. Solid (black) curves represent deviations from the intersection of the limit cycle (centered black dot) in terms of one and two standard deviations predicted using \eqref{eq:ex2pred} and \eqref{eq:eigslinper}. Dashed (blue) curves represent the one-standard deviation curve predicted from the statistical covariance matrix for the simulated data.}
    \label{fig3: po_ex_2}
\end{figure}

\section{Quasiperiodic Tori}
\label{sec: Quasiperiodic Tori}

A natural generalization of limit cycle dynamics is quasiperiodic motion on a transversally stable invariant torus, which we treat in this section. As in the previous section, we first derive a quasiperiodic covariance boundary-value problem to describe small fluctuations around a deterministic skeleton solution to leading order in the noise intensity. Next, we validate our approach against the results of stochastic numerical integration.

\subsection{Theoretical derivations}

Let $\mathbb{S}\sim[0,1]$. For some positive scalar\footnote{As in the previous section, for an autonomous vector field, $T$ is not typically known \textit{a priori} and is, instead, solved for simultaneously with $\gamma$ by appending suitable phase conditions to the quasiperiodic boundary-value problem in $x$ for fixed $\rho$, e.g., $h_i(\gamma(0,0))=0$ for some smooth scalar-valued functions $h_i$, $i=1,\ldots,m+1$. As before, we omit explicit parameter dependence for notational simplicity.} $T$, suppose that $\gamma(\phi,t)\in\mathbb{R}^n$ denotes a family of quasiperiodic solutions of the smooth dynamical system
\begin{equation}
\label{eq:determtorus}
\dot{x}=Tf(x),\qquad\mbox{ where }x:=x(t)\in\mathbb{R}^n,
\end{equation}
on an invariant torus, i.e., such that \eqref{eq:determtorus} is satisfied for $x(t)=\gamma(\phi,t)$ for all $\phi\in\mathbb{S}^m$, and there exists a function $u:\mathbb{S}^{m+1}\rightarrow\mathbb{R}^n$ and vector $\rho\in\mathbb{R}^m$ so that $r^\mathsf{T}\rho\ne 1$ for all rational vectors $r\in\mathbb{R}^m$ and $\gamma(\phi,t)=u(\phi+\rho t,t)$. In particular, $\gamma(\phi,t+1)=\gamma(\phi+\rho,t)$ for all $t$. When $m=0$, the analysis reduces to the case of a periodic orbit.

Let $X(\phi,t)$ denote the solution to the corresponding variational initial-value problem
\begin{equation}
\label{eq:varprobquasi}
\partial_t X=T\mathrm{D} f(\gamma)X,\quad X(\cdot,0)=I_n.
\end{equation}
By quasiperiodicity of $\gamma$, $X(\phi,t+1)=X(\phi+\rho,t)X(\phi,1)$. Furthermore, it follows that $X^{-1}(\cdot,t)\partial_t\gamma(\cdot,t)\equiv \partial_t\gamma(\cdot,0)$ and $X^{-1}(\cdot,t)\partial_{\phi_i}\gamma(\cdot,t)\equiv \partial_{\phi_i}\gamma(\cdot,0)$, $i=1,\ldots,m$, since equality holds at $t=0$ and differentiation with respect to $t$ followed by substitution of \eqref{eq:determtorus} and \eqref{eq:varprobquasi} shows that the left-hand sides are constant. As a special case, again by quasiperiodicity of $\gamma$, it follows that $\partial_t \gamma(\cdot,0)$ and $\partial_{\phi_i}\gamma(\cdot,0)$, $i=1,\ldots,m$, are right nullvectors of the operator
\begin{equation}
\Gamma_0:\delta(\cdot)\mapsto X(\cdot-\rho,1)\delta(\cdot-\rho)-\delta(\cdot)
\end{equation}
and that $\partial_t \gamma(\cdot,t)$ and $\partial_{\phi_i}\gamma(\cdot,t)$, $i=1,\ldots,m$, are right nullvectors of the operator
\begin{equation}
\Gamma_t:\delta(\cdot)\mapsto X(\cdot,t)\Gamma_0\left[X^{-1}(\cdot,t)\delta(\cdot)\right].
\end{equation}

Analogously with the periodic case and again following \cite{Dankowicz2022329}, assume next that $\gamma$ is \textit{normally hyperbolic} in the sense that there exists a quasiperiodic, smooth family of projections of the form
\begin{equation}
Q(\phi,t):=X(\phi,t)\left(I_n-\partial_t \gamma(\phi,0) w_t^\mathsf{T}(\phi)-\sum_{i=1}^{m}\partial_{\phi_i}\gamma(\phi,0) w_{\phi_i}^\mathsf{T}(\phi)\right)X^{-1}(\phi,t)
\end{equation}
for some periodic functions $w_t(\phi)$ and $w_{\phi_i}(\phi)$, $i=1,\ldots,m$, such that $\Gamma_0$ is a bijection with bounded inverse on the space of functions $\phi\mapsto Q(\phi,0)\delta(\phi)$ for arbitrary smooth functions $\delta$ on $\mathbb{S}$. As before, the family $Q(\phi,t)$ and the complementary family $I_n-Q(\phi,t)$ respect the spectral decomposition of $\Gamma_t+\mathrm{Id}$. 

In further analogy to the periodic case, we formally derive conditions on the nullspaces of $\Gamma_0$; the calculations are formal since we operate with transposition and integration instead of duality pairings. In particular, we let
\begin{equation}
    \int_{\mathbb{S}^n}w^\mathsf{T}(\phi)\left[\cdot\right]\,\mathrm{d}\phi
\end{equation}
denote a linear functional on the space of continuous functions on $\mathbb{S}$ in terms of some continuous, periodic function $w(\phi)$. Then, the left nullspace of $\Gamma_0$ are all functions $w(\phi)$ such that
\begin{equation}
    \int_{\mathbb{S}^n}w^\mathsf{T}(\phi)\Gamma_0[\delta(\phi)]\,\mathrm{d}\phi=0
\end{equation}
for arbitrary $\delta(\phi)$. Claims about uniqueness should be understood in this context as limited to spaces in which such representations are possible.

\subsubsection{The adjoint boundary-value problem}

We mirror the treatment in the previous section by deriving the adjoint boundary-value problems in Lemma~\ref{lem:adjbvpquasi} and a projection of the local dynamics onto directions transversal to the torus in Corollary~\ref{cor:transdynquasi}. The discussion is adapted from the treatment of two-tori ($m=1$) in \cite{Dankowicz2022329} .
\begin{lemma}
\label{lem:nullspacequa}
The linear functionals
\begin{equation}
\int_{\mathbb{S}^{m}}w_t^\mathsf{T}(\phi)\left[\cdot\right]\,\mathrm{d}\phi\mbox{ and }\int_{\mathbb{S}^{m}}w_{\phi_i}^\mathsf{T}(\phi)\left[\cdot\right]\,\mathrm{d}\phi,\,i=1,\ldots,m
\end{equation}
are the unique elements of the left nullspace of $\Gamma_0$ such that
\begin{equation}
\label{eq:normalize1}
w_t^\mathsf{T}(\cdot)\partial_t \gamma(\cdot,0)\equiv1,\,w_t^\mathsf{T}(\cdot)\partial_{\phi_i} \gamma(\cdot,0)\equiv0\equiv w_{\phi_i}^\mathsf{T}(\cdot)\partial_t\gamma(\cdot,0)
\end{equation}
for $i=1,\ldots,m$ and
\begin{equation}
\label{eq:normalize2}
w_{\phi_i}^\mathsf{T}(\cdot)\partial_{\phi_j} \gamma(\cdot,0)\equiv\delta_{ij}
\end{equation}
for $i,j=1,\ldots,m$, where $\delta_{ij}=1$ when $i=j$ and $0$ otherwise.
\end{lemma}
\begin{proof}
Let $\ast$ represents either $t$ or $\phi_i$. Then, quasiperiodicity $Q(\phi,1)=Q(\phi+\rho,0)$ holds if and only if $w_\ast^\mathsf{T}(\phi)X^{-1}(\phi,1)=w_\ast^\mathsf{T}(\phi+\rho)$. But,
\begin{align}
&\int_{\mathbb{S}^{m}}w_\ast^\mathsf{T}(\phi)\Gamma_0\left[\delta(\phi)\right]\,\mathrm{d}\phi=\int_{\mathbb{S}^{m}}w_\ast^\mathsf{T}(\phi)\left(X(\phi-\rho,1)\delta(\phi-\rho)-\delta(\phi)\right)\,\mathrm{d}\phi\nonumber\\
&\qquad\qquad=\int_{\mathbb{S}^{m}}\left(w_\ast^\mathsf{T}(\phi+\rho)X(\phi,1)-w_\ast^\mathsf{T}(\phi)\right)\delta(\phi)\,\mathrm{d}\phi
\end{align}
for an arbitrary periodic function $\delta(\phi)$. It follows that quasiperiodicity is equivalent to $\int_{\mathbb{S}^{m}}w_\ast^\mathsf{T}(\phi)\left[\cdot\right]\,\mathrm{d}\phi$ lying in the left nullspace of $\Gamma_0$. The conditions \eqref{eq:normalize1} and \eqref{eq:normalize2} follow by examining the equality $Q^2(\phi,0)=Q(\phi,0)$. Finally, to show uniqueness, suppose that
\begin{equation}
\int_{\mathbb{S}^{m}}\tilde{w}^\mathsf{T}(\phi)\left[\cdot\right]\,\mathrm{d}\phi
\end{equation}
is any left nullvector of $\Gamma_0$ and let $v(\phi)=Q(\phi,0)\delta(\phi)$ for an arbitrary continuous, periodic function $\delta(\phi)$. Since $\Gamma_0$ is onto functions of this form, there exists a periodic function $u(\phi)$ such that $\Gamma_0\left[u(\phi)\right]=v(\phi)$ and, consequently, 
\begin{equation}
\int_{\mathbb{S}^{m}}\tilde{w}^\mathsf{T}(\phi)Q(\phi,0)\delta(\phi)\,\mathrm{d}\phi=\int_{\mathbb{S}^{m}}\tilde{w}^\mathsf{T}(\phi)v(\phi)\,\mathrm{d}\phi=\int_{\mathbb{S}^{m}}\tilde{w}^\mathsf{T}(\phi)\Gamma_0\left[u(\phi)\right]\,\mathrm{d}\phi=0.
\end{equation}
But since $\delta$ is arbitrary, this implies that $\tilde{w}^\mathsf{T}(\phi)Q(\phi,0)\equiv 0$ and, consequently, that $\tilde{w}(\phi)\in\mathrm{span}\{w_t(\phi),w_{\phi_1}(\phi),\ldots,w_{\phi_m}(\phi)\}$ for each $\phi$. The claim follows with the imposition of \eqref{eq:normalize1} and \eqref{eq:normalize2}.
\end{proof}

\begin{corollary}
The right nullspace of $\Gamma_0$ consists of functions that are tangential to the manifold $(\phi,t)\mapsto\gamma(\phi,t)$ at every $\phi$ and $t=0$.
\end{corollary}
\begin{proof}
In the same notation as in Lemma~\ref{lem:nullspacequa},
\begin{equation}
X(\phi-\rho,1)\partial_\ast\gamma(\phi-\rho,0)w_\ast^\mathsf{T}(\phi-\rho)=\partial_\ast\gamma(\phi,0)w_\ast^\mathsf{T}(\phi)X(\phi-\rho,1)
\end{equation}
and, consequently, $\Gamma_t\left[Q(\phi,t)\cdot\right]=Q(\phi,t)\Gamma_t\left[\cdot\right]$. Now suppose that $\delta(\phi)$ is a right nullvector of $\Gamma_0$. It follows that $\Gamma_0\left[Q(\phi,0)\delta(\phi)\right]=Q(\phi,0)\Gamma_0\left[\delta(\phi)\right]=0$. Since $\Gamma_0$ is one-to-one on the space of functions $\phi\mapsto Q(\phi,0)\delta(\phi)$, $Q(\phi,0)\delta(\phi)\equiv 0$ and, consequently,
\begin{equation}
\delta(\phi)=\mathrm{span}\{\partial_t \gamma(\phi,0),\partial_{\phi_1}\gamma(\phi,0),\ldots,\partial_{\phi_m}\gamma(\phi,0)\},
\end{equation}
for each $\phi$.
\end{proof}
\begin{lemma}
\label{lem:adjbvpquasi}
The functions 
\begin{equation}
\lambda_t(\phi,t):=\left(X^{-1}\right)^\mathsf{T}(\phi,t)w_t(\phi)\mbox{ and }\lambda_{\phi_i}(\phi,t):=\left(X^{-1}\right)^\mathsf{T}(\phi,t)w_{\phi_i}(\phi)
\end{equation}
for $i=1,\ldots,m$ are the unique solutions to the adjoint boundary-value problem
\begin{equation}
\label{eq:covardifftorus}
0=-\partial_t \mu^\mathsf{T}-T\mu^\mathsf{T}\mathrm{D} f(\gamma),\,0=\mu(\phi+\rho,0)-\mu(\phi,1)
\end{equation}
where
\begin{equation}
1=\int_{\mathbb{S}^{m+1}}\mu^\mathsf{T}(\phi,t)\partial_t\gamma(\phi,t)\,\mathrm{d}t\,\mathrm{d}\phi
\end{equation}
and
\begin{equation}
0=\int_{\mathbb{S}^m}\mu^\mathsf{T}(\phi,1)\partial_{\phi_j}\gamma(\phi,1)\,\mathrm{d}\phi,\,j=1,\ldots,m
\end{equation}
in the case of $\lambda_t$, and 
\begin{equation}
0=\int_{\mathbb{S}^{m+1}}\mu^\mathsf{T}(\phi,t)\partial_t\gamma(\phi,t)\,\mathrm{d}t\,\mathrm{d}\phi
\end{equation}
and
\begin{equation}
\delta_{ij}=\int_{\mathbb{S}^m}\mu^\mathsf{T}(\phi,1)\partial_{\phi_j}\gamma(\phi,1)\,\mathrm{d}\phi,\,j=1,\ldots,m
\end{equation}
in the case of $\lambda_{\phi_i}$.
\end{lemma}
\begin{proof}
Let $\mu$ denote an arbitrary solution to the boundary-value problem. It follows that $\mu^\mathsf{T}(\phi,t)\partial_t\gamma(\phi,t)\equiv v_t(\phi)$, $\mu^\mathsf{T}(\phi,t)\partial_{\phi_i}\gamma(\phi,t)\equiv v_{\phi_i}(\phi)$, $i=1,\ldots,m$, and $\mu^\mathsf{T}(\phi,t)X(\phi,t)\equiv V(\phi)$ for some periodic functions $v_t(\phi)$, $v_{\phi_i}(\phi)$, $i=1,\ldots,m$, and $V(\phi)$, since differentiation with respect to $t$ followed by substitution of \eqref{eq:determtorus}, \eqref{eq:varprobquasi}, and \eqref{eq:covardifftorus} shows that the left-hand sides are independent of $t$. By quasiperiodicity,
\begin{equation}
v_t(\phi-k\rho)=\mu^\mathsf{T}(\phi,0)\partial_t\gamma(\phi,0)\mbox{ and }v_{\phi_i}(\phi-k\rho)=\mu^\mathsf{T}(\phi,0)\partial_{\phi_i}\gamma(\phi,0)
\end{equation}
for any integer $k$. Since $\rho$ is irrational, $v_t$ and $v_{\phi_i}$ must be independent of $\phi$. Additionally,
\begin{equation}
V^\mathsf{T}(\phi+\rho)X(\phi,1)-V^\mathsf{T}(\phi)=\mu^\mathsf{T}(\phi+\rho,0)X(\phi,1)-\mu^\mathsf{T}(\phi,1)X(\phi,1)=0
\end{equation}
and it follows that
\begin{equation}
\int_{\mathbb{S}^m}V^\mathsf{T}(\phi)\left[\cdot\right]\,\mathrm{d}\phi
\end{equation}
lies in the left nullspace of $\Gamma_0$.
\end{proof}
\begin{corollary}
In the notation of the previous lemma,
\begin{equation}
Q(\phi,t)=I_n-\partial_t \gamma(\phi,t) \lambda_t^\mathsf{T}(\phi,t)-\sum_{i=1}^{m}\partial_{\phi_i}\gamma(\phi,t) \lambda_{\phi_i}^\mathsf{T}(\phi,t)
\end{equation}
is a projection along $\mathrm{span}\{\partial_t \gamma(\phi,t), \partial_{\phi_1} \gamma(\phi,t),\ldots,\partial_{\phi_m} \gamma(\phi,t)  \}$ and onto a hyperplane orthogonal to $\lambda_t(\phi,t)$ and $\lambda_{\phi_i}(\phi,t)$ for $i=1,\ldots,m$.
\end{corollary}

We again investigate the geometry on a neighborhood of $\gamma$ in terms of the projections $Q$. These provide a local decomposition into state-space displacements tangential and transversal to the corresponding invariant torus.
\begin{lemma}
\label{lem:localqua}
For $x$ sufficiently close to $\gamma(\phi,t)$ for some $\phi$ and $t$, there exist unique $\psi$ (with $\psi\approx\phi$) and $\tau$ (with $\tau\approx t$), such that
\begin{equation}
x=\gamma(\psi,\tau)+x_\mathrm{tr},
\end{equation}
where $\lambda_t(\psi,\tau)^\mathsf{T}x_\mathrm{tr}=0$ and $\lambda_{\phi_i}(\psi,\tau)^\mathsf{T}x_\mathrm{tr}=0$ and, consequently, $Q(\psi,\tau)x_\mathrm{tr}=x_\mathrm{tr}$.
\end{lemma}
\begin{proof}
Let
\begin{equation}\label{eq_F_torus}
H(x,\psi,\tau)=\begin{pmatrix}\lambda_t^\mathsf{T}(\psi,\tau)\left(x-\gamma(\psi,\tau)\right)\\\lambda_{\phi_1}^\mathsf{T}(\psi,\tau)\left(x-\gamma(\psi,\tau)\right)\\\vdots\\\lambda_{\phi_m}^\mathsf{T}(\psi,\tau)\left(x-\gamma(\psi,\tau)\right)\end{pmatrix}
\end{equation}
It follows that $H(\gamma(\phi,t),\phi,t)=0$ and
\begin{equation}
\partial_{\psi,\tau} H(\gamma(\phi,t),\phi,t)=\begin{pmatrix}0  & -1\\-I_m & 0 \end{pmatrix}.
\end{equation}
The claim follows from the implicit function theorem.
\end{proof}

As in the previous section, we perform formal calculations with differentials in order to capture leading-order effects in the derivation of a suitable covariance function and the corresponding boundary-value problem.

\begin{corollary}
\label{cor:transdynquasi}
For the differential $\mathrm{d}x$, Lemma~\ref{lem:localqua} implies the existence of differentials $\mathrm{d}\tau$ and $\mathrm{d}x_\mathrm{tr}$, such that 
\begin{equation}
\label{eq:x+dxtorus}
x+\mathrm{d}x=\gamma(\psi+\mathrm{d}\psi,\tau+\mathrm{d}\tau)+x_\mathrm{tr}+\mathrm{d}x_\mathrm{tr},
\end{equation}
where $\lambda_t^\mathsf{T}(\psi+\mathrm{d}\psi,\tau+\mathrm{d}\tau)(x_\mathrm{tr}+\mathrm{d}x_\mathrm{tr})=0$ and $\lambda_{\phi_i}^\mathsf{T}(\psi+\mathrm{d}\psi,\tau+\mathrm{d}\tau)(x_\mathrm{tr}+\mathrm{d}x_\mathrm{tr})=0$ for $i=1,\ldots,m$, and, consequently,
\begin{equation}
\label{eq:Q(tau+dtau)torus}
Q(\psi+\mathrm{d}\psi,\tau+\mathrm{d}\tau)(x_\mathrm{tr}+\mathrm{d}x_\mathrm{tr})=x_\mathrm{tr}+\mathrm{d}x_\mathrm{tr}.
\end{equation}
It follows that
\begin{align}
\label{eq:dx_tr in Qtorus}
\mathrm{d}x_\mathrm{tr}&-T\mathrm{D} f(\gamma(\psi,\tau))x_\mathrm{tr}\mathrm{d}\tau-\sum_{i=1}^m\partial_{\phi_i}Q(\psi,\tau)x_\mathrm{tr}\mathrm{d}\psi_i\nonumber\\
&\qquad=Q(\psi,\tau)\left(\mathrm{d}x-T\mathrm{D} f(\gamma(\psi,\tau))x_\mathrm{tr}\mathrm{d}\tau\right).
\end{align}
\end{corollary}

\begin{proof}
Eq.~\eqref{eq:x+dxtorus} implies that
\begin{equation}
\label{eq:dx_tr1torus}
\mathrm{d}x_\mathrm{tr}=\mathrm{d}x-\partial_t\gamma(\psi,\tau)\mathrm{d}\tau-\sum_{i=1}^m\partial_{\phi_i}\gamma(\psi,\tau)\mathrm{d}\psi_i.
\end{equation}
Similarly, from \eqref{eq:Q(tau+dtau)torus}, we obtain
\begin{equation}
\label{eq:dx_tr2torus}
\mathrm{d}x_\mathrm{tr}=\partial_tQ(\psi,\tau)x_\mathrm{tr}\mathrm{d}\tau+\sum_{i=1}^m\partial_{\phi_i}Q(\psi,\tau)x_\mathrm{tr}\mathrm{d}\psi_i+Q(\psi,\tau)\mathrm{d}x_\mathrm{tr}.
\end{equation}
From the definition of $Q$ it follows that
\begin{align}
\partial_tQ(\psi,\tau)&=-T\mathrm{D} f(\gamma(\psi,\tau)) \left(\partial_t\gamma(\psi,\tau)\lambda_t^\mathsf{T}(\psi,\tau)+\sum_{i=1}^m\partial_{\phi_i}\gamma(\psi,\tau)\lambda_{\phi_i}^\mathsf{T}(\psi,\tau)\right)\nonumber\\
&\qquad+T\left(\partial_t\gamma(\psi,\tau)\lambda_t^\mathsf{T}(\psi,\tau)+\sum_{i=1}^m \partial_{\phi_i}\gamma(\psi,\tau)\lambda_{\phi_i}^\mathsf{T}(\psi,\tau)\right)\mathrm{D} f(\gamma(\psi,\tau))
\end{align}
and, consequently, that
\begin{equation}
\label{eq:Qdotx_trtorus} 
\partial_tQ(\psi,\tau)x_\mathrm{tr}=T(I_n-Q(\psi,\tau))\mathrm{D} f(\gamma(\psi,\tau))x_\mathrm{tr},
\end{equation}
since $\lambda_t^\mathsf{T}(\psi,\tau)x_\mathrm{tr}=0$ and $\lambda_{\phi_i}^\mathsf{T}(\psi,\tau)x_\mathrm{tr}=0$ for $i=1,\ldots,m$. Substitution of \eqref{eq:dx_tr1torus} and \eqref{eq:Qdotx_trtorus} into the right-hand side of \eqref{eq:dx_tr2torus} yields \eqref{eq:dx_tr in Qtorus}, since $Q(\psi,\tau)\partial_t\gamma(\psi,\tau)=0$ and $Q(\psi,\tau)\partial_{\phi_i}\gamma(\psi,\tau)=0$ for $i=1,\ldots,m$.
\end{proof}

\subsubsection{The covariance boundary-value problem}
In lieu of the deterministic dynamics in \eqref{eq:determtorus}, consider again the It\^{o} SDE
\begin{equation}
\label{eq:sde}
\mathrm{d}x=Tf(x)\mathrm{d}t+\sigma \sqrt{T}F(x)\mathrm{d}W_{t}
\end{equation} 
in terms of the noise intensity $\sigma$, and a vector $W_{t}\in\mathbb{R}^m$ of independent standard Brownian motions. In the notation of the preceding lemmas, and as before under a small noise assumption and up to leading-order terms, the following proposition holds.

\begin{proposition}
For $x$ sufficiently close to $\gamma(\phi,t)$ for some $\phi$ and $t$, to first order in $\|x_\mathrm{tr}\|$ and $\sigma$,
\begin{equation}
\label{eq:OUprocesstorus}
\mathrm{d}x_\mathrm{tr}=T\mathrm{D} f(\gamma(\psi,\tau))x_\mathrm{tr}\mathrm{d}\tau+\sigma \sqrt{T}Q(\psi,\tau)F(\gamma(\psi,\tau))\mathrm{d}W_{\tau}.
\end{equation}
\end{proposition}

\begin{proof}
For $x_\mathrm{tr}=0$ and $\sigma=0$, we obtain $\tau=t$, $\psi=\phi$, and $\mathrm{d}x_\mathrm{tr}=0$. In this limit, \eqref{eq:sde} and \eqref{eq:dx_tr1torus} imply that
\begin{equation}
0=\partial_t\gamma(\psi,\tau)(\mathrm{d}t-\mathrm{d}\tau)-\sum_{i=1}^m\partial_{\phi_i}\gamma(\psi,\tau)\mathrm{d}\psi_i,
\end{equation}
from which it follows that $\mathrm{d}\tau=\mathrm{d}t$ and $\mathrm{d}\psi_i=0$, $i=1,\ldots,m$, to zeroth order in $\|x_\mathrm{tr}\|$ and $\sigma$. To first order in $\|x_\mathrm{tr}\|$ and $\sigma$, \eqref{eq:sde} then yields
\begin{equation}
Q(\psi,\tau)\left(\mathrm{d}x-T\mathrm{D} f(\gamma(\psi,\tau))x_\mathrm{tr}\mathrm{d}\tau\right)=\sigma\sqrt{T}Q(\psi,\tau) F(\gamma(\psi,\tau))\mathrm{d}W_\tau
\end{equation}
and substitution in \eqref{eq:dx_tr in Qtorus} yields the claimed result.
\end{proof}

The Ornstein-Uhlenbeck process  in \eqref{eq:OUprocesstorus} approximates the local dynamics near $\gamma$ for orbits that remain near $\gamma$ to leading order but neglects cumulative effects of higher-order terms, e.g., differences between $\mathrm{d}t$ and $\mathrm{d}\tau$ and higher-order terms in $\mathrm{d}\psi$. Given an initial condition $x_\mathrm{tr}(\psi,0)$, \eqref{eq:OUprocesstorus} may be solved explicitly to yield
\begin{align}
x_\mathrm{tr}(\psi,\tau)&=X(\psi,\tau)x_\mathrm{tr}(\psi,0)+\sigma\sqrt{T}X(\psi,\tau)\int_{0}^{\tau}X^{-1}(\psi,s)Q(\psi,s)F(\gamma(\psi,s)))\mathrm{d}W_{s}\nonumber\\
    &=X(\psi,\tau)x_\mathrm{tr}(\psi,0)+\sigma\sqrt{T}Q(\psi,\tau)X(\psi,\tau)\int_{0}^{\tau}G(\psi,s)\mathrm{d}W_{s},
\end{align}
where $G(\psi,s)=X^{-1}(\psi,s)F(\gamma(\psi,s))$. With
\begin{equation}
\Lambda(\psi,\tau):=\begin{pmatrix}\lambda_t(\psi,\tau) & \lambda_{\phi_1}(\psi,\tau) & \cdots & \lambda_{\phi_m}(\psi,\tau)\end{pmatrix}
\end{equation}
and
\begin{equation}
\Omega(\psi):=\begin{pmatrix}w_t(\psi) & w_{\phi_1}(\psi) & \cdots & w_{\phi_m}(\psi)\end{pmatrix}
\end{equation}
it follows that $\Lambda^\mathsf{T}(\psi,\tau)x_\mathrm{tr}(\psi,\tau)\equiv 0$ provided that $\Omega^\mathsf{T}(\psi)x_\mathrm{tr}(\psi,0)=0$. The rescaled covariance matrix
\begin{equation}
\label{eq:cov_deftorus}
    C:=\frac{1}{\sigma^2}\mathbb{E}\left[ x_\mathrm{tr}x_\mathrm{tr}^{\mathrm{T}}\right]
\end{equation}
is then given by
\begin{align}
\label{eq:Ctr_integraltorus}
C(\psi,\tau)&=X(\psi,\tau)C(\psi,0)X^\mathsf{T}(\psi,\tau)\nonumber\\
&\qquad +TQ(\psi,\tau)X(\psi,\tau)\left(\int_{0}^{\tau}G(\psi,s)G^\mathsf{T}(\psi,s)\mathrm{d}s\right)X^\mathsf{T}(\psi,\tau)Q^\mathsf{T}(\psi,\tau),
\end{align}
It follows that $\Lambda^\mathsf{T}(\psi,\tau)C(\psi,\tau)\Lambda(\psi,\tau)$ is independent of $\tau$ and equals $0$ if and only if $\Omega(\psi)^\mathsf{T}C(\psi,0)\Omega(\psi)=0$.
\begin{lemma}
Let
\begin{equation}
\mathcal{I}(\psi):=\int_0^1G(\psi,s)G^\mathsf{T}(\psi,s)\mathrm{d}s
\end{equation}
Then, for arbitrary $C(\psi,0)$,
\begin{align}
&C(\psi,k+1)=X(\psi+k\rho,1)C(\psi,k)X^\mathsf{T}(\psi+k\rho,1)\nonumber\\
&\qquad+TQ(\psi+k\rho,1)X(\psi+k\rho,1)\mathcal{I}(\psi+k\rho)X^\mathsf{T}(\psi+k\rho,1)Q^\mathsf{T}(\psi+k\rho,1).
\end{align}
\end{lemma}
\begin{proof}
By \eqref{eq:Ctr_integraltorus},
\begin{align}
&C(\psi,k)=X(\psi,k)C(\psi,0)X^\mathsf{T}(\psi,k)\nonumber\\
&\qquad+TQ(\psi,k)X(\psi,k)\left(\int_0^{k}G(\psi,s)G^\mathsf{T}(\psi,s)\mathrm{d}s\right)X^\mathsf{T}(\psi,k)Q^\mathsf{T}(\psi,k).
\end{align}
Moreover, by the definition of $G$,
\begin{equation}
\mathcal{I}(\psi+k\rho)=X(\psi,k)\left(\int_k^{k+1}G(\psi,s)G^\mathsf{T}(\psi,s)\mathrm{d}s\right)X^\mathsf{T}(\psi,k).
\end{equation}
The claim follows by substitution.
\end{proof}

We proceed to assume that the quasiperiodic orbit is \textit{transversally stable}, such that
\begin{align}
X(\psi,k)&=\partial_t\gamma(\psi+k\rho,0))w_t^\mathsf{T}(\psi)\nonumber\\
&\qquad+\sum_{i=1}^m\partial_{\phi_i}\gamma(\psi+k\rho,0))w_{\phi_i}^\mathsf{T}(\psi)+\mathcal{O}\left(\mathrm{exp}(-k/\tau_{\mathrm{tr}})\right)
\end{align}
for some positive constant $\tau_\mathrm{tr}$ and large $k$. Then,
\begin{proposition}
Given a constant matrix $B$, there exists an initial condition $C(\psi,0)$ with $\Omega(\psi)^\mathsf{T}C(\psi,0)\Omega(\psi)=B$, such that $C(\psi,k)=C(\psi+k\rho,0)$ for all $k$.
\end{proposition}
\begin{proof}
The previous lemma shows that $C(\psi,1)=C(\psi+\rho,0)\Rightarrow C(\psi,k)=C(\psi+k\rho,0)$ for all $k$. Suppose that
\begin{align}
\label{eq:C0explicittorus}
C(\psi,0)&=T\sum_{n=1}^\infty Q(\psi,0)X(\psi-n\rho,n)\mathcal{I}(\psi-n\rho)X^\mathsf{T}(\psi-n\rho,n)Q^\mathsf{T}(\psi,0)\nonumber\\
&\quad+\nabla(\psi,0)B\nabla^\mathsf{T}(\psi,0).
\end{align}
where
\begin{equation}
\nabla(\psi,\tau):=\begin{pmatrix}\partial_t\gamma(\psi,\tau) & \partial_{\phi_1}\gamma(\psi,\tau) & \cdots & \partial_{\phi_m}\gamma(\psi,\tau)\end{pmatrix}
\end{equation}
The series converges by the assumption of transversal stability. Moreover,
\begin{align}
&X(\psi,1)C(\psi,0)X^\mathsf{T}(\psi,1)=C(\psi+\rho,0)\nonumber\\
&\qquad-TQ(\psi+\rho,0)X(\psi,1)\mathcal{I}(\psi)X^\mathsf{T}(\psi,1)Q^\mathsf{T}(\psi+\rho,0)
\end{align}
and the claim follows by evaluation of \eqref{eq:Ctr_integraltorus} at $\tau=1$.
\end{proof}
\begin{corollary}
Given a constant matrix $B$, the function $C(\psi,\tau)$ in \eqref{eq:Ctr_integraltorus} with $\Omega(\psi)^\mathsf{T}C(\psi,0)\Omega(\psi)=\beta$ converges to the function obtained with $C(\psi,0)$ given by \eqref{eq:C0explicittorus} as $\tau\rightarrow\infty$.
\end{corollary}

\begin{proposition}
\label{prop: qp_covar_bv}
Let $C_\mathrm{qper}$ denote the function in \eqref{eq:Ctr_integraltorus} obtained with $C(\psi,0)$ given by \eqref{eq:C0explicittorus} when $B=0$. It follows that, when the matrix $A=0$, $C_\mathrm{qper}$ is the unique quasiperiodic solution of the quasiperiodic Lyapunov equation (cf.~\cite{Halanay1987})
\begin{align}
\partial_tC&=T\left(\mathrm{D} f(\gamma)C+C\mathrm{D} f^\mathsf{T}(\gamma)+QF(\gamma)F^\mathsf{T}(\gamma)Q^\mathsf{T}+\nabla A\nabla^\mathsf{T}\right),
\end{align}
for which $\Omega(\psi)^\mathsf{T}C(\psi,0)\Omega(\psi)=0$ and that no quasiperiodic solution exists otherwise.
\end{proposition}
\begin{proof}
It follows from \eqref{eq:covardifftorus} that
\begin{equation}
\frac{d}{dt}\Lambda(\psi,t)^\mathsf{T}C(\psi,t)\Lambda(\psi,t)=A T
\end{equation}
Thus, $C$ is quasiperiodic only if $A=0$. The transformation $C=X\tilde{C}X^\mathsf{T}$ then implies that
\begin{equation}
\partial_t\tilde{C}(\psi,t)=TQ(\psi,0)G(\psi,t)G^\mathsf{T}(\psi,t)Q^\mathsf{T}(\psi,0)
\end{equation}
and the conclusion follows by integration.
\end{proof}

We conclude by noting that
\begin{equation}
    \Omega(0)^\mathsf{T}C(0,0)\Omega(0)=0\Rightarrow \Lambda^\mathsf{T}(0,k)C(0,k)\Lambda(0,k)=0
\end{equation}
Assuming quasiperiodicity, it follows that
\begin{equation}
    \Omega(k\rho)^\mathsf{T}C(k\rho,0)\Omega(k\rho)=0
\end{equation}
for every $k$. Since $\rho$ is irrational and by continuity, it follows that $    \Omega(\phi)^\mathsf{T}C(\phi,0)\Omega(\phi)=0$ for all $\phi$.

\subsection{Numerical examples}\label{sec:numexpqua}
As a first example, we take inspiration from the second example in Sec.~\ref{sec:numexpper} and consider the non-autonomous, two-dimensional system of SDEs
\begin{align}
    \mathrm{d}x_1&=\frac{2\pi}{\omega}\left(-\Omega x_2+x_1\left(1+\sqrt{x_1^2+x_2^2}(\cos2\pi t-1)\right)\right)\mathrm{d}t+\sigma\sqrt{\frac{2\pi}{\omega}}x_1x_2\mathrm{d}W_t\\
    \mathrm{d}x_2&=\frac{2\pi}{\omega}\left(\Omega x_1+x_2\left(1+\sqrt{x_1^2+x_2^2}(\cos2\pi t-1)\right)\right)\mathrm{d}t+\sigma\sqrt{\frac{2\pi}{\omega}}x_2^2\mathrm{d}W_t
\end{align}
of the form \eqref{eq:sde} with $T=2\pi/\omega$, $f$ given by the $1$-periodic vector field
\begin{equation}
    f(t,x)=\begin{pmatrix}-\Omega x_2+x_1\left(1+\sqrt{x_1^2+x_2^2}(\cos2\pi t-1)\right)\\\Omega x_1+x_2\left(1+\sqrt{x_1^2+x_2^2}(\cos2\pi t-1)\right)\end{pmatrix},
\end{equation}
and $F(x)=(x_1x_2,x_2^2)^\mathsf{T}$. Rather than consider the suspended autonomous state space obtained by appending a third component to the vector field equal to $1$, we retain the two-dimensional description but omit $w_t$ and $\lambda_t$ in the definition of $Q$.

A quasiperiodic solution of the deterministic dynamics is given by
\begin{equation}
    \gamma(\phi,t)=\frac{1+\omega^2}{1+\omega^2-\cos2\pi t-\omega\sin2\pi t}\begin{pmatrix}
    \cos2\pi(\phi+\rho t)\\\sin2\pi(\phi+\rho t)
    \end{pmatrix}
\end{equation}
provided that $\rho=\Omega/\omega$ is irrational. With the notation $b:=1+\omega^2-\cos2\pi t-\omega\sin2\pi t$, $c_{\phi,t}:=\cos2\pi(\phi+\rho t)$, and $s_{\phi,t}:=\sin2\pi(\phi+\rho t)$, we obtain
\begin{equation}
    X(\phi,t)=\frac{\omega^2}{b^2}\begin{pmatrix}e^{-2\pi t/\omega}\omega^2c_{\phi,0} c_{\phi,t}+bs_{\phi,0} s_{\phi,t} & e^{-2\pi t/\omega}\omega^2s_{\phi,0} c_{\phi,t}-bc_{\phi,0} s_{\phi,t}\\ e^{-2\pi t/\omega}\omega^2c_{\phi,0} s_{\phi,t}-bs_{\phi,0} c_{\phi,t} & e^{-2\pi t/\omega}\omega^2s_{\phi,0} s_{\phi,t}+bc_{\phi,0} c_{\phi,t} \end{pmatrix}
\end{equation}
such that, in particular, $\partial_\phi\gamma(\phi,t)=X(\phi,t)\partial_\phi\gamma(\phi,0)$ while it no longer holds that $\partial_t\gamma(\phi,t)=X(\phi,t)\partial_t\gamma(\phi,0)$, since $f$ is non-autonomous. It is now straightforward to verify the equalities $w_\phi^\mathsf{T}(\phi)X^{-1}(\phi,1)=w_\phi^\mathsf{T}(\phi+\rho)$ and $w_\phi^\mathsf{T}(\phi)\partial_\phi\gamma(\phi,0)\equiv 1$ for
\begin{equation}
    w_\phi(\phi)=\frac{\omega^2}{2\pi(1+\omega^2)}\begin{pmatrix}-\sin2\pi\phi\\\cos2\pi\phi\end{pmatrix}
\end{equation}
and we obtain the family of projections
\begin{equation}
    Q(\phi,t)=X(\phi,t)\left(I_2-\partial_\phi\gamma(\phi,0)w_\phi^\mathsf{T}(\phi)\right)X^{-1}(\phi,t)=\begin{pmatrix}
    c_{\phi,t}^2 & s_{\phi,t}c_{\phi,t}\\s_{\phi,t}c_{\phi,t} & s_{\phi,t}^2
    \end{pmatrix}
\end{equation}
onto the radial direction at each point of the corresponding torus. Substitution in \eqref{eq:C0explicittorus} (with $B=0$), followed by extensive algebraic manipulation using \textsc{Mathematica}, then yields the family
\begin{equation}
\label{eq:quasiex1pred}
    C(\phi,0)=\frac{\left(1+\omega ^2\right)^4 \left(1+\Omega ^2-c_{2\phi,0}-\Omega  s_{2\phi,0}\right)}{4\omega
   ^8 \left(1+\Omega ^2\right)}Q(\phi,0).
\end{equation}
In Fig.~\ref{fig4: qp_ex_1}, we compare this prediction to density clouds of points along stochastic trajectories sampled at $t\in\mathbb{Z}$ and observe excellent agreement.

\begin{figure}[htp]
    \centering
    {\includegraphics[width=0.5\textwidth]{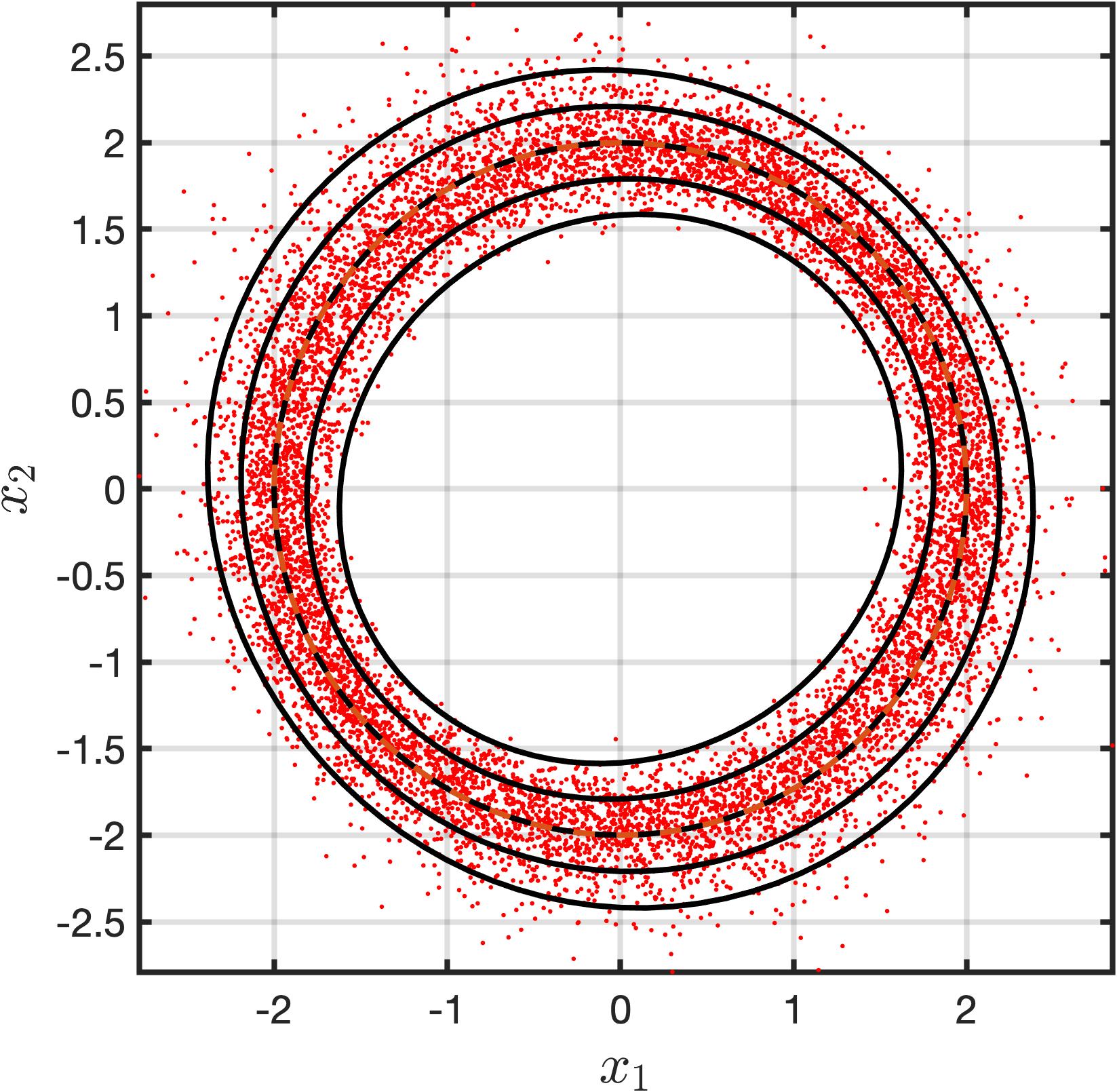}}
    \bigskip
    
    {\includegraphics[width=0.6\textwidth]{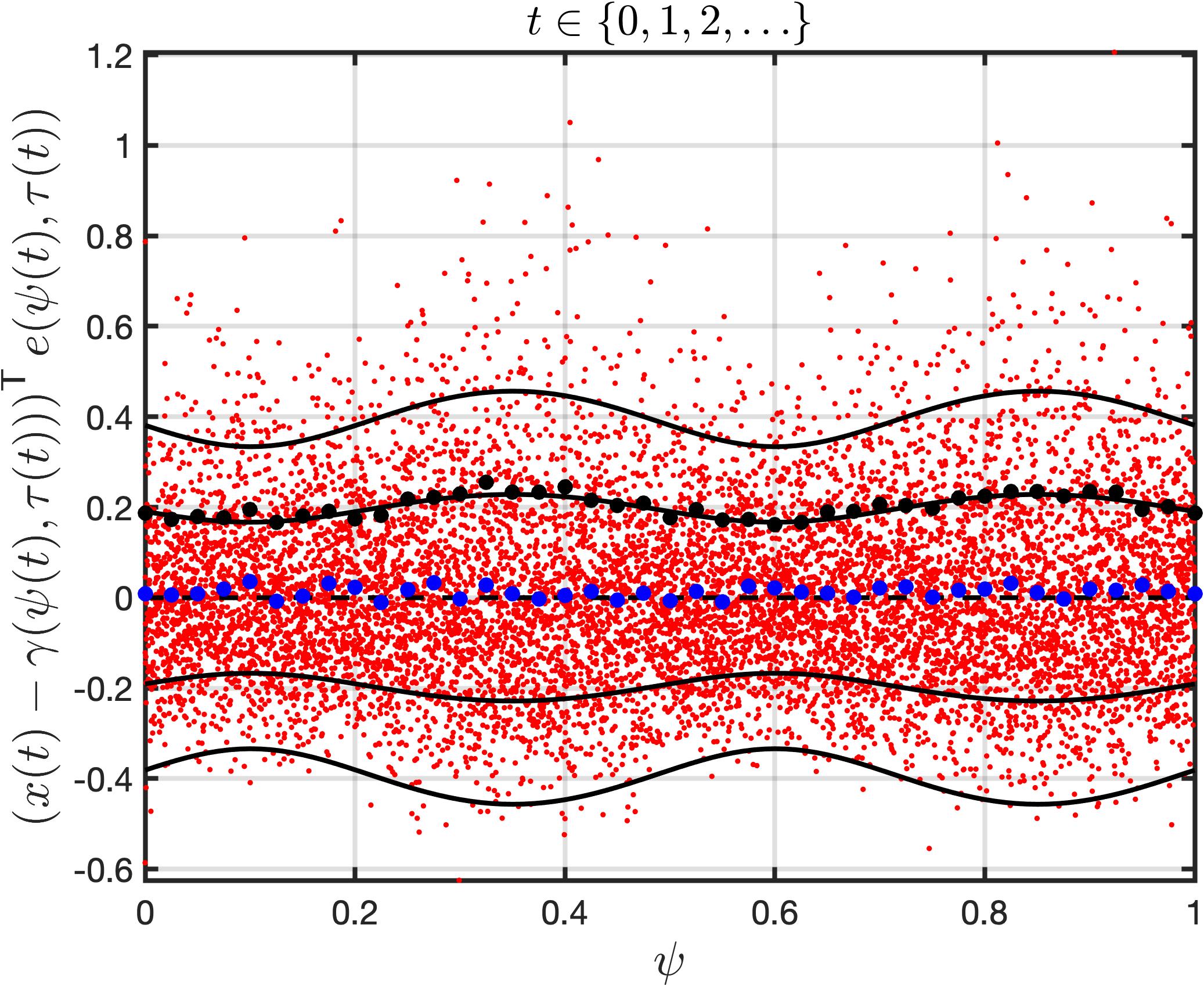}}
    \caption{Comparison of leading-order theoretical predictions and a numerical time history for the first example in Sec.~\ref{sec:numexpqua} with $\Omega=\pi$ and $\omega=1$ in cartesian (upper panel) and polar (lower panel) coordinates. Stochastic trajectories sampled at integer values of the excitation phase (red dots) were obtained using an Euler-Maruyama scheme with $\sigma=0.1$ and $\mathrm{d}t=10^{-4}$ for $10,000$ periods of the excitation. Dashed curves represent the intersection of the deterministic quasiperiodic invariant torus while solid curves represent predicted deviations from this curve of intersection equal to one and two standard deviations, respectively, computed using \eqref{eq:quasiex1pred}. In the lower panel, the vertical axis equals the radial deviation from the (circular) curve of intersection, since the normalized radial eigenvector $e(\psi,0)\parallel\gamma(\psi,0)$ and $w_\phi(\psi)^\mathsf{T}\left(x-\gamma(\psi,0)\right)=0$. There, filled circles represent the statistical mean (blue) and standard deviation (black) of simulated data collected in bins of width $\Delta\psi=1/40$.}
    \label{fig4: qp_ex_1}
\end{figure}

As a second example, this time not amenable to closed-form analysis, we consider the four-dimensional SDE obtained by adding noise to two coupled Van der Pol oscillators:
\begin{align}
\mathrm{d}y=\left(\begin{array}{c}
y_{2}\\
-\epsilon\left(y_{1}^{2}-1\right)y_{2}-y_{1}+\beta\left(y_{3}-y_{1}\right)\\
y_{4}\\
-\epsilon\left(y_{3}^{2}-1\right)y_{4}-\left(1+\delta\right)y_{3}+\beta\left(y_{1}-y_{3}\right)
\end{array}\right)\mathrm{d\eta}+\sigma\begin{pmatrix}0\\\mathrm{d}W_{\eta,1}\\0\\\mathrm{d}W_{\eta,2}
\end{pmatrix}.\label{eq:tori_sde}
\end{align}
Here $\epsilon,\delta,\beta$ are the problem parameters and $W_{\eta}\in \mathbb{R}^{2}$ is a vector of independent standard Brownian motions. Quasiperiodic invariant tori of the deterministic drift vector field have been investigated by Schilder \textit{et al.}~\cite{Schilder2005459}.

In the limit of small $\epsilon$ and $\beta=0$, a quasiperiodic solution of the rescaled deterministic dynamics with $T=2\pi$ is approximately given by
\begin{equation}
    \gamma(\phi,t)=\begin{pmatrix}
    2\sin 2\pi t\\2\cos 2\pi t\\2\sin 2\pi(\phi+\rho t)\\2\rho\cos 2\pi(\phi+\rho t)
    \end{pmatrix}
\end{equation}
for irrational $\rho=\sqrt{1+\delta}$. We may use this expression and corresponding value of $T$ to construct an initial solution guess for quasiperiodic solutions of the rescaled deterministic dynamics for nonzero $\epsilon$ with $\beta=0$. Parameter continuation may then be used to track a family of such solutions for fixed $\rho$ and simultaneous variations of $\beta$ and $\delta$. An example of such a family for $\rho=140/62\sqrt{2}$ and $\epsilon=0.5$ (computed with \textsc{coco} using the discretization algorithm described in the appendix with $N=14$ and individual trajectory segments approximated by continuous piecewise-polynomial functions on 20 mesh intervals) together with a projection of the invariant torus obtained for $\beta=0.5$ are shown in Fig.~\ref{fig5: qp_ex_2}.

\begin{figure}[ht!]
    \centering
    {\includegraphics[width=0.65\textwidth]{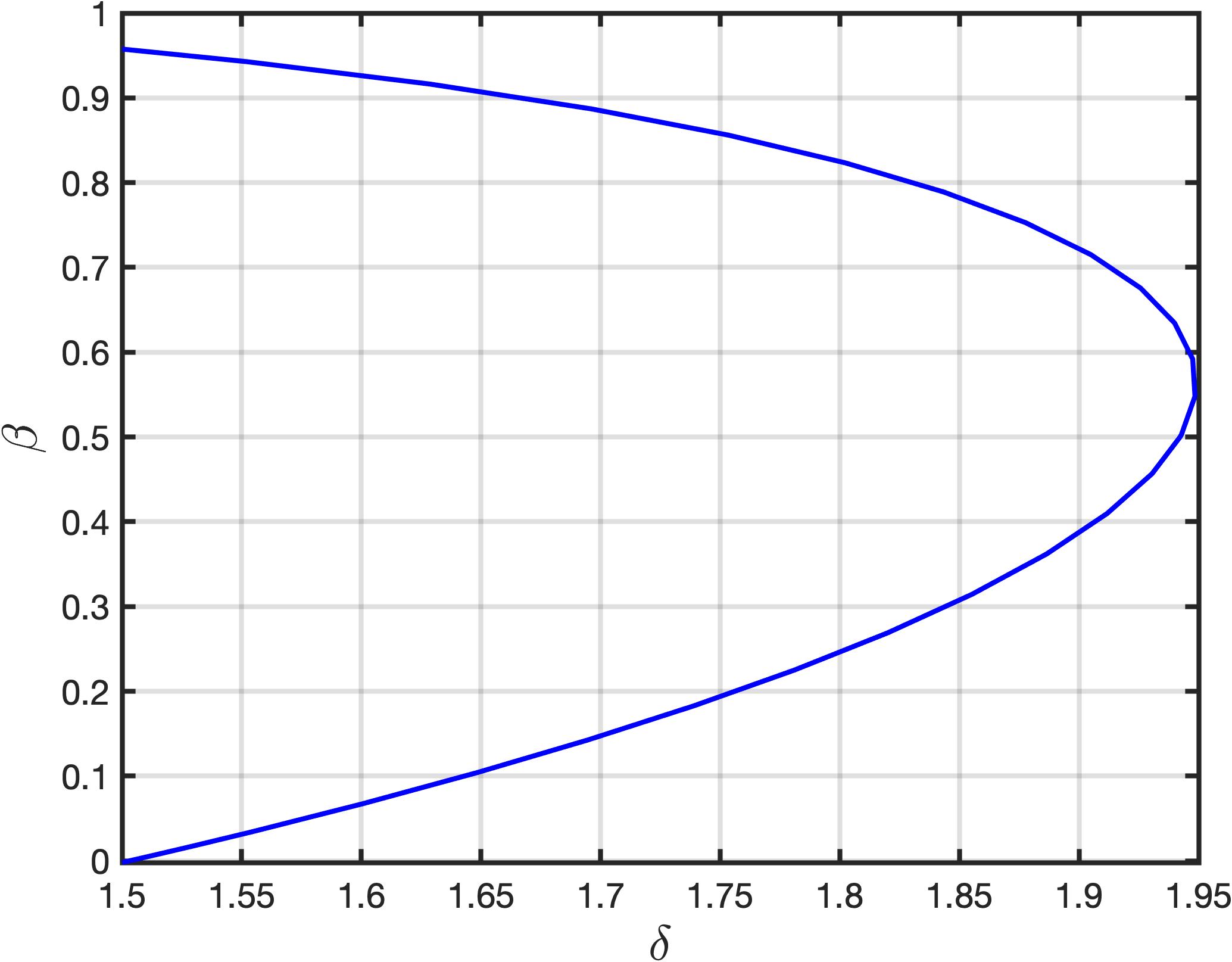}}
    \bigskip
    
    {\includegraphics[width=0.65\textwidth]{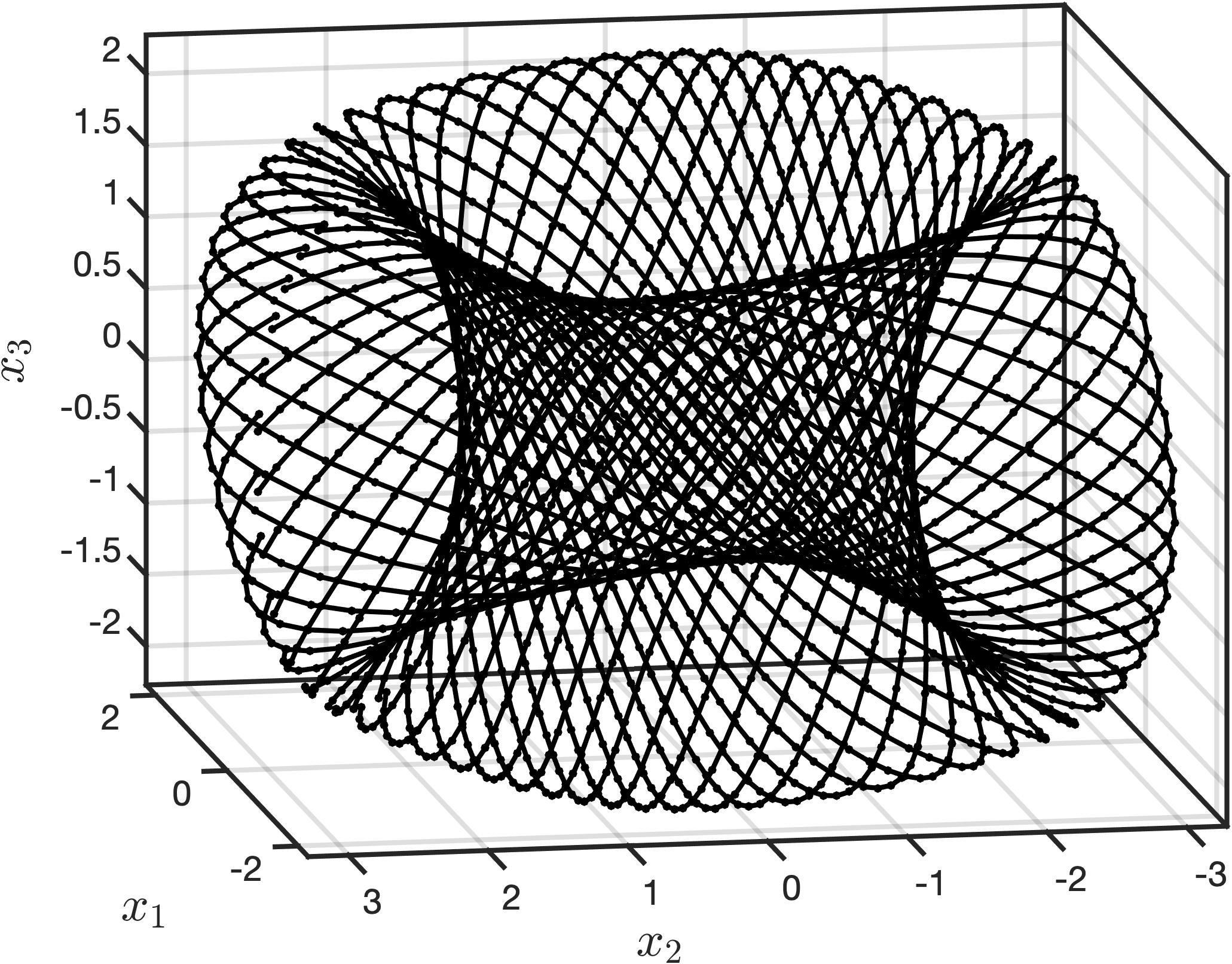}}
    \caption{(Upper panel) Family of quasiperiodic invariant tori for the deterministic limit of the SDE~\eqref{eq:tori_sde} under simultaneous variations in $\beta$ and $\delta$ computed using \textsc{coco}. Here, $\epsilon=0.5$ and the rotation number $\rho$ is fixed at $140/62\sqrt{2}$. Each torus is represented by a finite collection of equal-duration trajectory segments coupled through all-to-all boundary conditions in terms of the rotation number. (Lower panel) The quasiperiodic invariant torus obtained for $\beta=0.5$ represented in terms of 29 trajectory segments, each of which is approximated by a continuous, piecewise-polynomial function on 20 mesh intervals.}
    \label{fig5: qp_ex_2}
\end{figure}
 
In the absence of closed-form analysis, we use built-in support in \textsc{coco} for the corresponding discretization of the adjoint boundary-value problem (see the Appendix) to compute numerical approximations of $\lambda_t(\phi,t)$, $\lambda_\phi(\phi,t)$, and $Q(\phi,t)$ for the torus obtained for $\beta=0.5$. Finally, we derive a corresponding discretization of the linear boundary-value problem in Proposition~\ref{prop: qp_covar_bv} and solve for the covariance matrix function $C(\phi,t)$. We use the norm of the computed matrix $A$ to verify the accuracy of the computation, as this is predicted to equal $0$ in the original boundary-value problem. The results of this analysis are visualized in Fig.~\ref{fig6: qp_ex_2_eigs}. In this case, the 2-norm of $A$ equals $9\times 10^{-9}$, giving us confidence in the validity of the analysis. 

\begin{figure}[htbp]
    \centering
    {\includegraphics[width=0.65\textwidth]{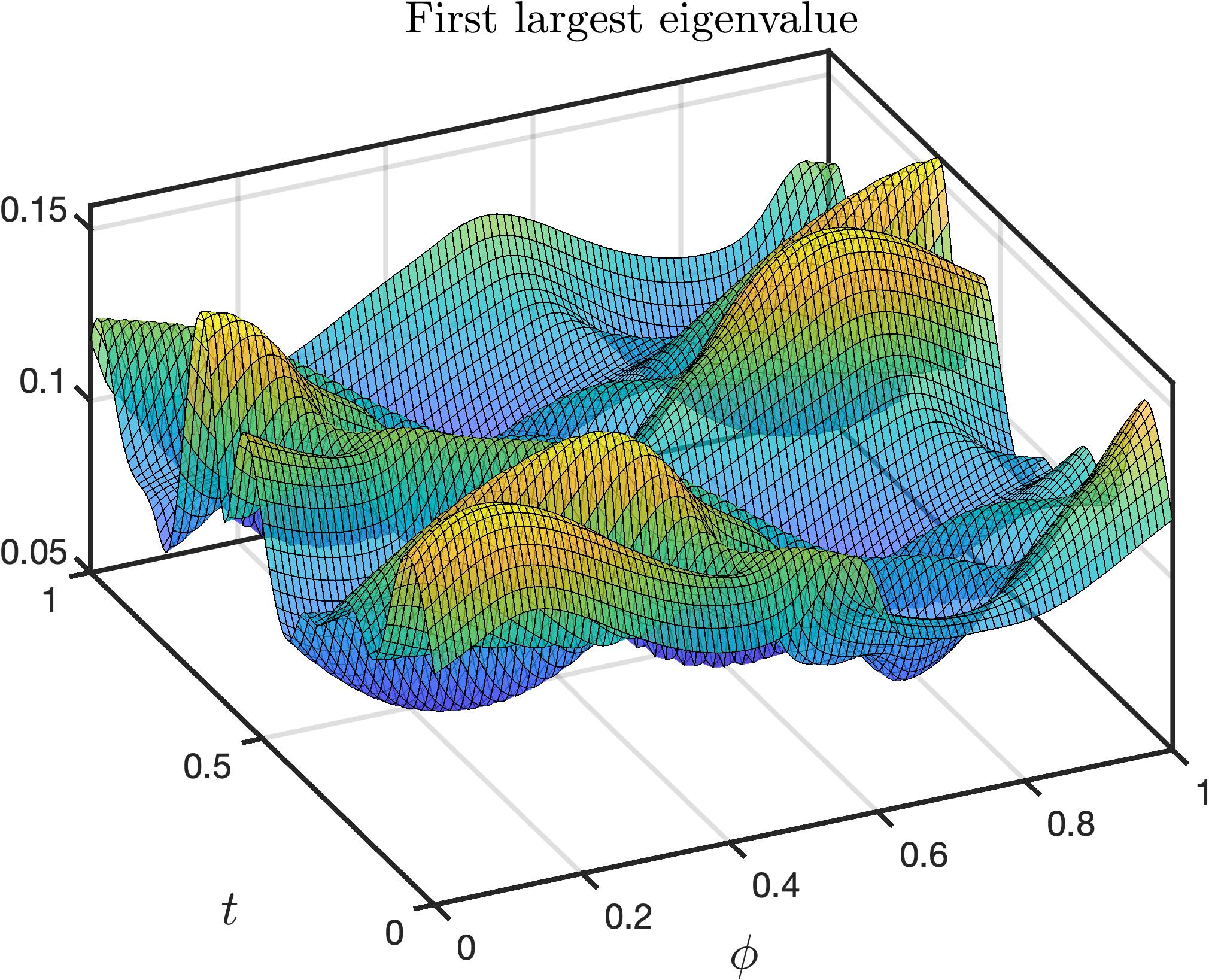}}
    \bigskip
    
    {\includegraphics[width=0.65\textwidth]{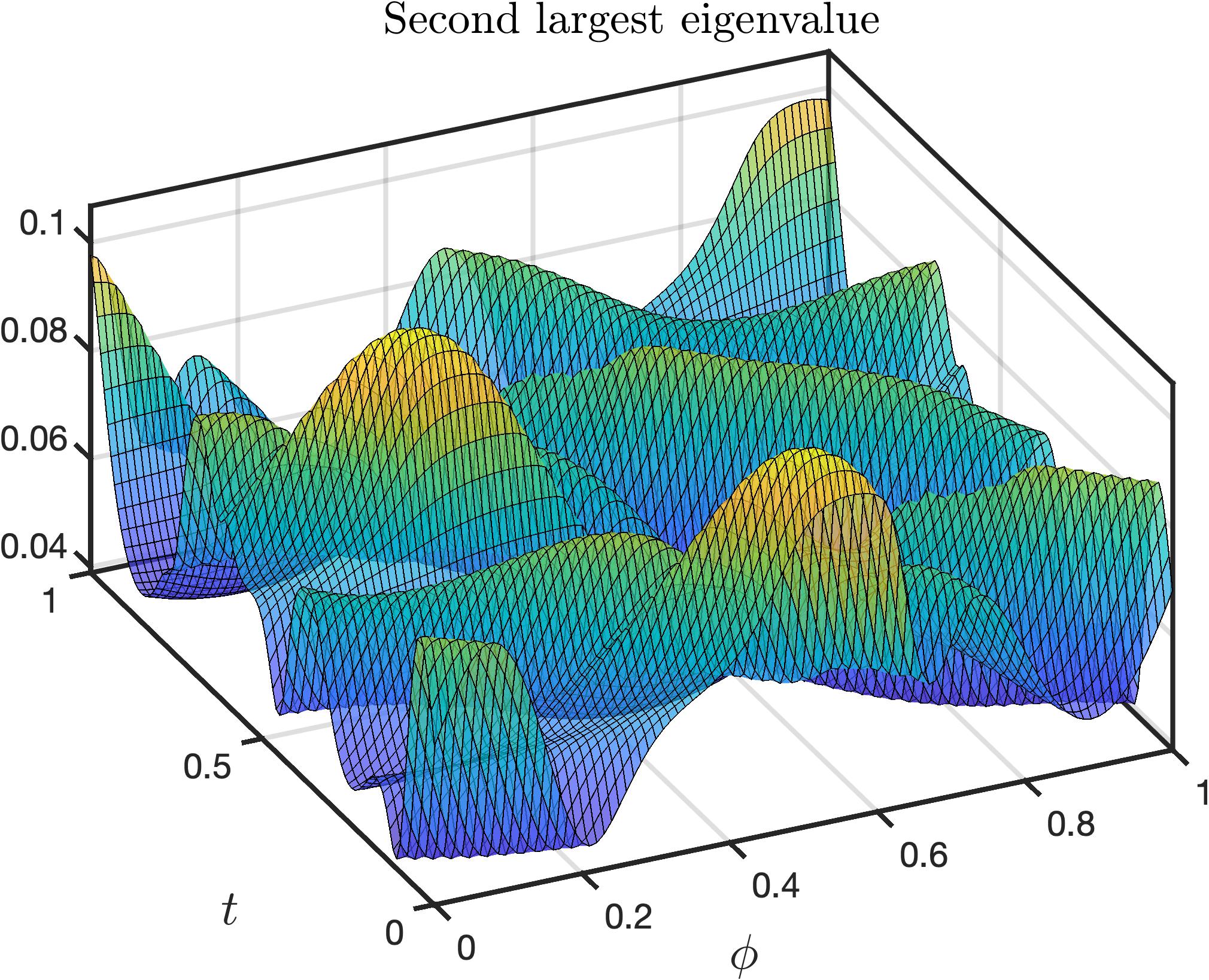}}
    \caption{Surfaces of the two nonzero eigenvalues for the covariance matrix function $C(\phi,t)$ corresponding to the quasiperiodic invariant torus obtained for the deterministic limit of the SDE~\eqref{eq:tori_sde} for $\epsilon=0.5$, $\beta=0.5$, $\delta=1.9422$ and the rotation number $\rho=140/62\sqrt{2}$.}
    \label{fig6: qp_ex_2_eigs}
\end{figure}

As in the previous examples, we use points along numerically integrated stochastic trajectories to validate the theoretical predictions. For points on such trajectories, we first use the constraints in Lemma~\ref{lem:localqua} to identify the corresponding values of $\psi$ and $\tau$. We then proceed to locate ``intersections'' with $\tau=0.5$ by finding consecutive trajectory points on opposite sides of this section and use linear interpolation through these points to capture a point on $\tau=0.5$ and compute the corresponding value of $\psi$ (see Fig.~\ref{fig7: qp_ex_2_section}). Finally, we project the resultant deviations $x_\mathrm{tr}$ onto the predicted eigenvectors of the corresponding covariance matrix associated with the two nonzero eigenvalues (see Fig.~\ref{fig8: qp_ex_2_projs}). The agreement between the statistics for the resultant point clouds and the predicted standard deviations is excellent.

\begin{figure}[htbp]
    \centering
    {\includegraphics[width=0.75\textwidth]{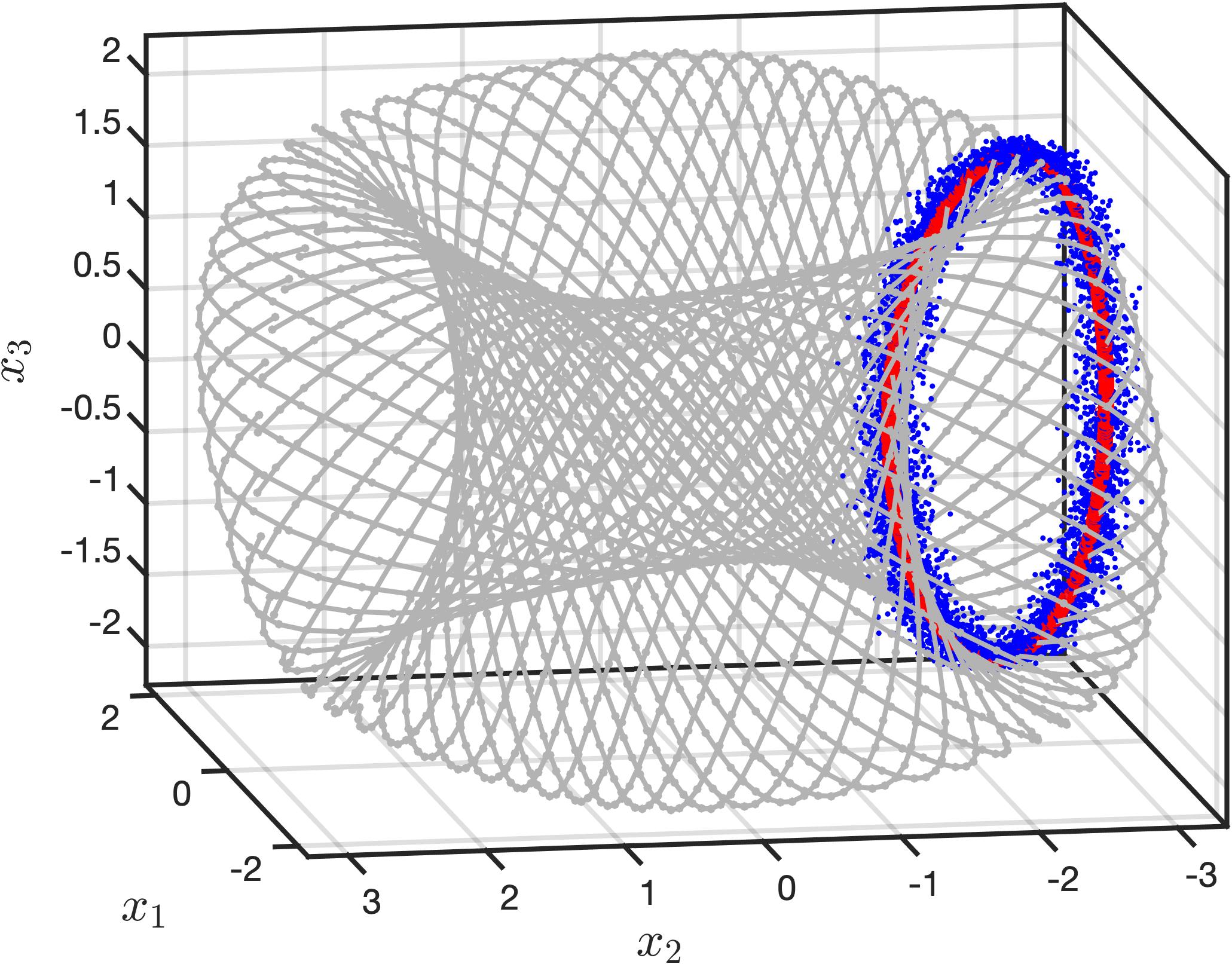}}
    \caption{A point cloud of $5,800$ ``intersections'' (blue) along a numerically integrated stochastic trajectory of the SDE~\eqref{eq:tori_sde} with the section $\tau=0.5$ overlaid on the quasiperiodic invariant torus for the corresponding deterministic dynamics obtained for $\beta=0.5$ (see the lower panel of Fig.~\ref{fig5: qp_ex_2}) and the corresponding points $\gamma(\psi,0.5)$ (red). Stochastic trajectories were obtained using an Euler-Maruyama scheme with $\sigma=0.1$ and $\mathrm{d}t=10^{-4}$ and associated with points on the torus per the method described in the text.}
    \label{fig7: qp_ex_2_section}
\end{figure}

\begin{figure}[htbp]
    \centering
    \includegraphics[width=0.65\textwidth]{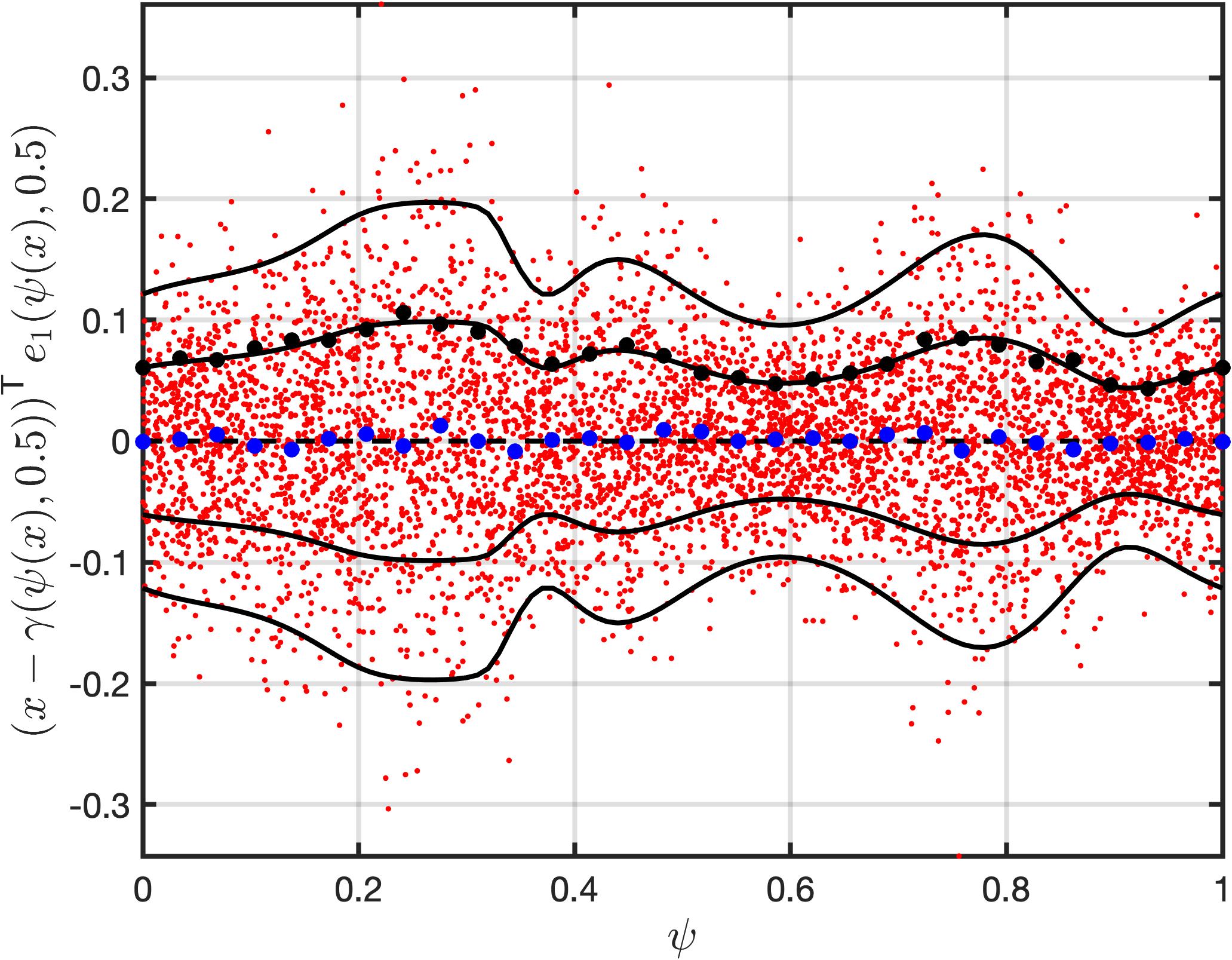}
    \bigskip
    
    \includegraphics[width=0.65\textwidth]{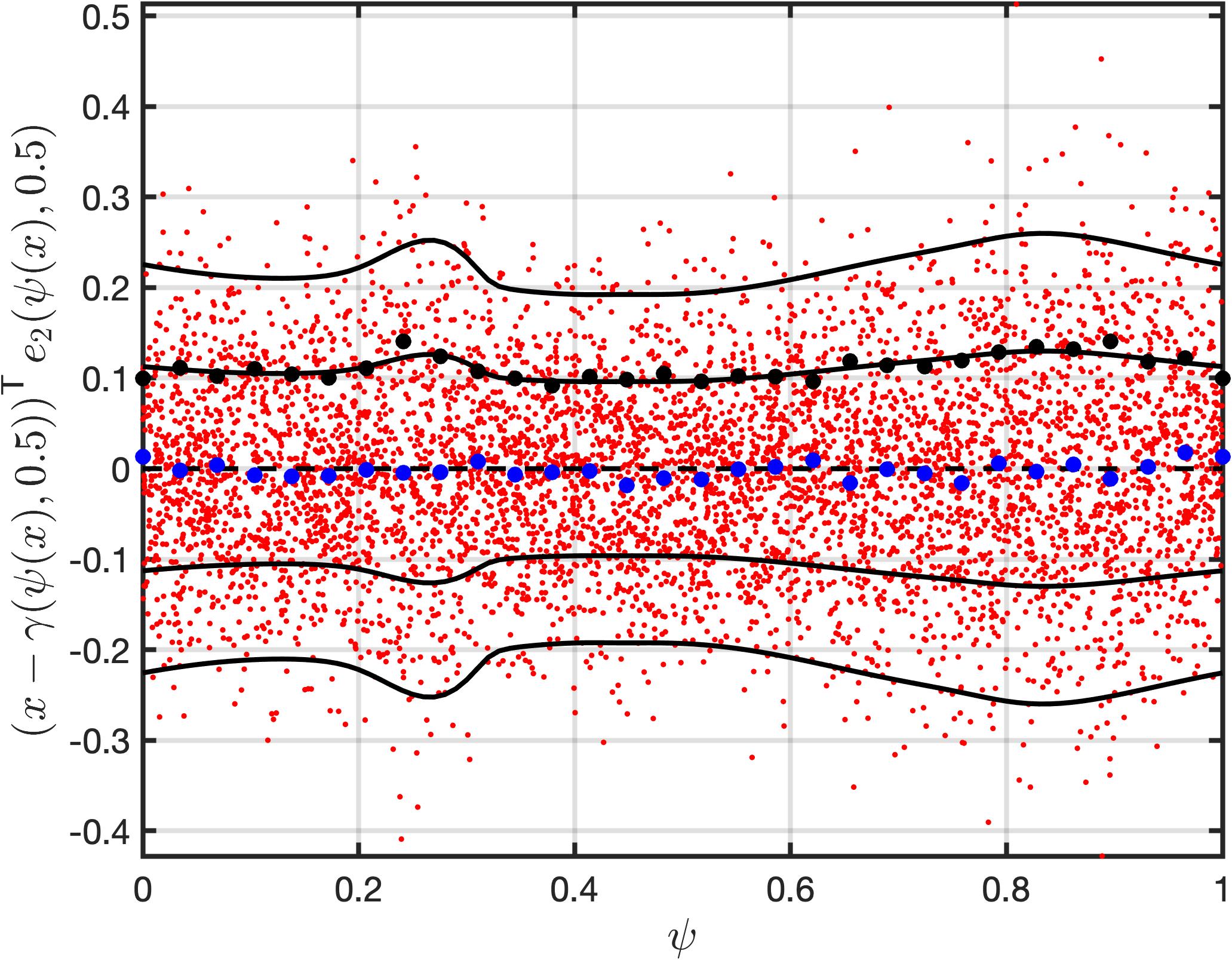}
    \caption{Comparison of leading-order theoretical predictions and the results of analysis of the point cloud in Fig.~\ref{fig7: qp_ex_2_section} obtained for the SDE~\eqref{eq:tori_sde}. For each point $x$ in this cloud, the deviation $x-\gamma(\psi(x),0.5)$ is projected onto the normalized eigenvectors $e_1(\psi(x),0.5)$ and $e_2(\psi(x),0.5)$ corresponding to the two nonzero eigenvalues of the predicted covariance matrix $C(\psi(x),0.5)$. Dashed curves represent points on the deterministic quasiperiodic invariant torus while solid curves represent predicted deviations from this curve of intersection equal to one and two standard deviations, respectively, obtained by multiplying the square roots of the eigenvalues by $\sigma=0.1$. Filled circles represent the statistical mean (blue) and standard deviation (black) of the projected simulated data collected in bins of width $\Delta\psi=1/29$.}
    \label{fig8: qp_ex_2_projs}
\end{figure}

\section{Conclusions}
\label{sec: Conclusions}
The results derived here provide a computationally attractive alternative to and generalization of the earlier work by Bashkirtseva and Ryashko~\cite{bashkirtseva2016sensitivity}, Guo \emph{et al.}~\cite{Guo2017}, and Zhao \emph{et al.}~\cite{Zhao2022}. Where these references rely on projections onto hyperplanes orthogonal to the vector field, transversal to the vector field but otherwise arbitrary, and stroboscopic sections in a case of periodically excitations on top of an existing limit cycle motion, the proposed boundary-value problem formulation is closely integrated with a foliation of transversal hyperplanes that may be defined invariantly through a defining boundary-value problem, without the need for an explicit coordinate basis or additional approximation. From the invariance of this foliation under the linearized deterministic flow, there further follows a closed-form solution for the covariance function in the form of a rapidly convergent series for autonomous as well as periodically excited systems.

A further advancement is the explicit recognition of the existence of $(m+1)^2$ conserved quantities for the forward dynamics of the covariance differential equation with unique periodic or quasiperiodic solutions in each corresponding level surface. Following a recipe analogous to that developed by Mu\~{n}oz-Almaraz \emph{et al.}~\cite{MUNOZALMARAZ20031}, we removed this degeneracy by imposing a finite set of transversality constraints and including symmetry-breaking parameters collected in a matrix $A$ (scalar $a$ for the periodic orbit case) which must equal $0$ on the sought solution. As described in conjunction with an implementation of a suitable discretization of the covariance boundary-value problem in the software package \textsc{coco}, the norm of $A$ provides a useful measure of the truncation error. Notably, this is not available in our alternative implementation that relies on the Moore-Penrose pseudo-inverse to solve the overconstrained boundary-value problem with $A$ set to $0$ \emph{a priori}.

The availability of three examples for which analytical expressions may be found for the adjoint variables and covariance matrix functions is another useful contribution of this paper. As the authors can attest, such analytical expressions are invaluable for debugging computational software that seeks to make predictions about problems with uncertainty, especially for eliminating various sources of systematic bias. In the present study, they highlighted an early erroneous assumption about the symmetry of the matrices $B$ and $A$, and led to the final imposition of the full set of $(m+1)^2$ transversality conditions $\Omega(0)^\mathsf{T}C(0,0)\Omega(0)=0$.

A numerical solver for the covariance boundary-value problem was strictly only needed for the second example in Section~\ref{sec:numexpqua}. Nevertheless, our \textsc{coco} implementation relies on a general-purpose routine for one- (limit cycles) and two-dimensional tori that is compatible with and may be appended to the corresponding nonlinear continuation problem for the torus together with the associated adjoint boundary-value problems. Although the final implementation solves for the covariance function only after the torus and the adjoint functions have been obtained, it is possible to embed the covariance boundary-value problem with the full continuation problem. This comes at the expense of a larger problem dimension but with the benefit that conditions may be imposed on its solution. For example, it might be of interest to track families of local maxima of the eigenvalues of the covariance function under variations in problem parameters or to hold such local maxima fixed during continuation.

Finally, although the covariance boundary-value problem has been investigated here as a means to characterize noise-induced statistics near transversally stable periodic orbits of quasiperiodic invariant tori, we expect that it may serve a useful, but different purpose also in the unstable, but normally hyperbolic case. For such a periodic orbit or quasiperiodic torus, a solution to the corresponding covariance boundary-value problem no longer captures the statistics of trajectories that remain near the deterministic limit cycle over long times, since none do. With suitably chosen noise, however, we expect that such a solution could be used as an alternative means of quantifying the stability of the deterministic object. We are particularly interested in exploring this in the quasiperiodic case, as a substitute for analysis of the solution to the variational problem.

\section*{Acknowledgments}
This material is based upon work performed while the second author (HD) was supported by and served at the National Science Foundation. Any opinion, findings, and conclusions or recommendations expressed in this material are those of the authors and do not necessarily reflect the views of the National Science Foundation. CK also wants to thank the VolkswagenStiftung for support via a Lichtenberg Professorship.

\appendix
\section{Adjoint Conditions and Problem Discretization}
\label{app: Adjoint Conditions and Problem Discretization}

It is, of course, the rare exception when a closed form analysis is both possible and expeditious. In all other cases, the boundary-value problems in the main body of this paper must be implemented numerically in order to compute $\gamma$, $\lambda$ or $\Lambda$, and $C$ and explore variations in $C$ along continuous families of such solutions.

Of the three boundary-value problems, the one for the adjoint $\lambda$ and $\Lambda$ stand out in that the differential equation and periodic or quasiperiodic boundary conditions are supplemented by normalization conditions involving integrals. We show here that the adjoint boundary-value problems, including the integral conditions, may be obtained in stages from a calculus-of-variations approach applied to a suitable constraint Lagrangian (as discussed in~\cite{li2018staged}). We rely on this when using the built-in problem constructors in \textsc{coco} to derive regular zero problems for discretized approximations of $\gamma$ and $\lambda$ or $\Lambda$.

The following two propositions and subsequent corollary review results from \cite{Dankowicz2022329}. Original to this paper is Proposition~\ref{prop:discrete} which translates directly to the \textsc{coco} implementation used to compute $\gamma$, $\Lambda$, and $C$ for the second example in Sec.~\ref{sec:numexpqua}.
\begin{proposition}
Suppose that $h(\gamma(0))=0$ and $\partial_x h(\gamma(0))f(\gamma(0))\ne 0$ for some function $h$. Then, the boundary-value problem
\begin{equation}
0=-\dot{\mu}^\mathsf{T}-T\mu^\mathsf{T}\mathrm{D} f(\gamma),\,0=\mu(0)-\mu(1),\,1=\int_0^1\mu^\mathsf{T}f(\gamma)\,\mathrm{d}t
\end{equation}
is equivalent to the adjoint conditions obtained by imposing vanishing variations of
\begin{equation}
L=S+\int_0^1\mu^\mathsf{T}\left(\dot{x}-Sf(x)\right)\,\mathrm{d}t+w^\mathsf{T}(x(0)-x(1))+\kappa h(x(0))
\end{equation}
under variations in $x$ and $S$, when evaluated at $x=\gamma$ and $S=T$.
\end{proposition}
\begin{proof}
The integral condition results directly from vanishing variations with respect to $S$. Vanishing variations with respect to $x$ yield the differential equation and the additional conditions
\begin{equation}
0=\mu(1)^\mathsf{T}-w,\,0=-\mu(0)+w+\kappa\partial_x h(x(0))
\end{equation}
Multiplication of both equations by $f(\gamma(0))$ and use of periodicity and the constancy of $\mu^\mathsf{T}f(\gamma)$ then yields $\kappa=0$ and the claim follows.
\end{proof}

Recall the notation
\begin{equation}
\nabla(\phi,t):=\begin{pmatrix}\partial_t\gamma(\phi,t) & \partial_{\phi_1}\gamma(\phi,t) & \cdots & \partial_{\phi_m}\gamma(\phi,t)\end{pmatrix}
\end{equation}
from the section on quasiperiodic orbits.
\begin{proposition}
\label{prop: quasi}
Suppose that $h_i(\gamma(0,0))=0$, $i=1,\ldots,m+1$ and that the matrix
\begin{equation}
\begin{pmatrix}
\label{eq:transversality}
\partial_x h_1(\gamma(0,0)) \partial_t \gamma(0,0) & \cdots & \partial_x h_{m+1}(\gamma(0,0)) \partial_t \gamma(0,0)\\\vdots & \ddots & \vdots\\\partial_x h_1(\gamma(0,0)) \partial_{\phi_m} \gamma(0,0) & \cdots & \partial_x h_{m+1}(\gamma(0,0)) \partial_{\phi_m} \gamma(0,0)\end{pmatrix}
\end{equation}
is invertible. Then, the boundary-value problem
\begin{align}
0=-\dot{\mu}^\mathsf{T}&-T\mu^\mathsf{T}\mathrm{D} f(\gamma(\phi,t)),0=\mu(\phi+\rho,0)-\mu(\phi,1),\nonumber\\
&\begin{pmatrix}1 & 0 & \cdots & 0\end{pmatrix}=\int_{\mathbb{S}^m}\mu^\mathsf{T}\nabla\,\mathrm{d}\phi
\end{align}
is equivalent to the adjoint conditions obtained by imposing vanishing variations of
\begin{align}
L&=\ln S+\int_{\mathbb{S}^m}\int_0^1\mu^\mathsf{T}\left(\dot{x}-Sf(x)\right)\,\mathrm{d}t\,\mathrm{d}\phi\nonumber\\
&\qquad +\int_{\mathbb{S}^m}w^\mathsf{T}(\phi)(x(\phi+\varrho,0)-x(\phi,1))\,\mathrm{d}\phi+\sum_{i=1}^{m+1}\kappa_i h_i (x(0,0))
\end{align}
under variations in $x$, $S$, and $\varrho$, when evaluated at $x=\gamma$, $S=T$, and $\varrho=\rho$.
\end{proposition}
\begin{proof}
Vanishing variations with respect to $x$ yield the differential equation and the additional conditions
\begin{equation}
0=\mu^\mathsf{T}(\phi,1)-w^\mathsf{T}(\phi)
\end{equation}
and
\begin{equation}
0=-\mu^\mathsf{T}(\phi,0)+w^\mathsf{T}(\phi-\rho)+\delta_\mathrm{D}(\phi)\sum_{i=1}^{m+1}\kappa_i\partial_x h_i (\gamma(0,0))
\end{equation}
in terms of the Dirac delta function $\delta_\mathrm{D}$.
Multiplication of the first condition from the right by $\nabla(\phi,1)$, of the second condition from the right by $\nabla(\phi,0)$, addition of the results using the constancy of $\mu^\mathsf{T}\nabla$ with respect to $t$ to eliminate terms including $\mu$, and integration over $\mathbb{S}^m$  using quasiperiodicity of $\nabla$ then yields $\kappa_i=0$, $i=1,\ldots,m+1$, thus establishing the quasiperiodicity of $\mu$. The integral conditions now follow from variations with respect to $S$ and $\varrho$.
\end{proof}
\begin{corollary}
\label{cor:lambdaphi}
By replacing the term $\ln S$ in the expression for $L$ by $-\varrho_i$, one obtains the integral conditions associated with $\lambda_{\phi_i}$.
\end{corollary}

One notes, in the quasiperiodic case, the absence of partial derivatives with respect to components of $\phi$ in either of the differential equations for $\gamma$, $\Lambda$, or $C$. As described in \cite{dankowicz2013recipes}, this suggest a discretization of all $\phi$-dependent unknowns in terms of truncated Fourier series such that, e.g.,
\begin{equation}
\label{eq:Four1}
\gamma(\phi,t)=a_0(\phi_{-1},t)+\sum_{k=1}^Na_k(\phi_{-1},t)\cos 2\pi k\phi_1+b_k(\phi_{-1},t)\sin 2\pi k\phi_1
\end{equation}
for some integer $N$ and where $\phi_{-k}$ denotes the argument sequence $\phi_{k+1},\ldots,\phi_m$. 
\begin{lemma}
Let $\phi_{1,j}:=(j-1)/(2N+1)$ and define the notation
\begin{equation}
B(\phi_1):=\begin{pmatrix}1 & \cos 2\pi\phi_1 & \sin 2\pi\phi_1 & \cdots & \cos 2N\pi\phi_1 & \sin 2N\pi\phi_1\end{pmatrix}
\end{equation}
and
\begin{equation}
G(\phi_{-1},t):=\begin{pmatrix}\gamma_1(\phi_{-1},t)\\\vdots\\\gamma_{2N+1}(\phi_{-1},t)\end{pmatrix},
\end{equation}
where $\gamma_j(\phi_{-1},t):=\gamma(\phi_{1,j},\phi_{-1},t)$. Then, there exists an invertible $(2N+1)\times(2N+1)$ matrix $\mathcal{F}$ such that
\begin{equation}
\gamma(\phi,t)=\left(B(\phi_1)\mathcal{F}\otimes I_n\right)G(\phi_{-1},t)
\end{equation}
and
\begin{equation}
\mathcal{F}^\mathsf{T}\begin{pmatrix}1 & & &\\& 1/2 & &\\& &  \ddots &\\& &  & 1/2\end{pmatrix}\mathcal{F}=\frac{1}{2N+1}I_{2N+1}.
\end{equation}
\end{lemma}
\begin{proof}
By inspection, $\mathcal{F}^{-1}=$
\begin{equation}
{\small \begin{pmatrix}
1 & \cos 2\pi\phi_{1,1} & \sin 2\pi\phi_{1,1} & \cdots & \cos 2\pi N\phi_{1,1} & \sin 2\pi N\phi_{1,1}\\
\vdots & \vdots & \vdots & \ddots & \vdots & \vdots\\
1 & \cos 2\pi\phi_{1,2N+1} & \sin 2\pi\phi_{1,2N+1} & \cdots & \cos 2\pi N\phi_{1,2N+1} & \sin 2\pi N\phi_{1,2N+1}\end{pmatrix}
}
\end{equation}
and the claim follows by properties of the trigonometric functions.
\end{proof}
\begin{corollary}
Let
\begin{equation}
\mathcal{R}(\rho_1):=\begin{pmatrix}
1 & & & & &\\
& \cos2\pi\rho_1 & \sin2\pi\rho_1 & & &\\& -\sin2\pi\rho_1 & \cos2\pi\rho_1 & & &\\& & & \ddots & &\\
&& & & \cos2\pi N\rho_1 & \sin2\pi N\rho_1\\& && & -\sin2\pi N\rho_1 & \cos2\pi N\rho_1
\end{pmatrix}
\end{equation}
and
\begin{equation}
\mathcal{D}:=2\pi\begin{pmatrix}
0 & & & & &\\
& 0 & 1 & & &\\& -1 & 0 & & &\\& & & \ddots & &\\
&& & & 0 & N\\& && & -N & 0
\end{pmatrix}
\end{equation}
Then, $B'(\phi_1)=B(\phi_1)\mathcal{D}$, $\mathcal{R}'(\rho_1)=\mathcal{R}(\rho_1)\mathcal{D}$, and
\begin{equation}
\gamma(\phi+\rho,t)=\left(B(\phi_1)\mathcal{R}(\rho_1)\mathcal{F}\otimes I_n\right)G(\phi_{-1}+\rho_{-1},t),
\end{equation}
\end{corollary}
\begin{proposition}
\label{prop:discrete}
Let
\begin{equation}
M(\phi_{-1},t):=\begin{pmatrix}\mu(\phi_{1,1},\phi_{-1},t)\\\vdots\\\mu(\phi_{1,2N+1},\phi_{-1},t)\end{pmatrix}.
\end{equation}
The discretized form of the adjoint boundary-value problem in Proposition~\ref{prop: quasi} consists of the differential equations
\begin{equation}
0=-\partial_tM^\mathsf{T}-TM^\mathsf{T}\begin{pmatrix}\mathrm{D} f(\gamma_1) & & \\ & \ddots & \\& & \mathrm{D} f(\gamma_{2N+1}) \end{pmatrix}
\end{equation}
the boundary conditions
\begin{equation}
0=\left(\mathcal{F}\otimes I_n\right)M(\phi_{-1},1)-\left(\mathcal{R}(\rho_1)\mathcal{F}\otimes I_n\right)M(\phi_{-1}+\rho_{-1},0),
\end{equation}
and the integral conditions
\begin{equation}
1=\frac{1}{2N+1}\int_{\mathbb{S}^{m-1}}M^\mathsf{T}\partial_tG\,\mathrm{d}\phi_{-1},
\end{equation}
\begin{equation}
0=\frac{1}{2N+1}\int_{\mathbb{S}^{m-1}}M^\mathsf{T}\partial_{\phi_i}G\,\mathrm{d}\phi_{-1},
\end{equation}
for $i=2,\ldots,m$, and
\begin{equation}
0=\int_{\mathbb{S}^{m-1}}M^\mathsf{T}\left(\mathcal{F}^\mathsf{T}\mathcal{D}^\mathsf{T}\begin{pmatrix}1 & & &\\& 1/2 & &\\& &  \ddots &\\& &  & 1/2\end{pmatrix}\mathcal{F}\otimes I_n\right)G\,\mathrm{d}\phi_{-1},
\end{equation}
These are equivalent to the adjoint conditions obtained by imposing vanishing variations of
\begin{align}
L&=(2N+1)\ln S+\int_{\mathbb{S}^{m-1}}\int_0^1M^\mathsf{T}\left(\partial_t\Xi-SF\right)\,\mathrm{d}t\,\mathrm{d}\phi_{-1}\nonumber\\
&+\int_{\mathbb{S}^{m-1}}w^\mathsf{T}(\phi_{-1})\bigg(\left(\mathcal{R}(\varrho_1)\mathcal{F}\otimes I_n\right)\Xi(\phi_{-1}+\varrho_{-1},0)-\left(\mathcal{F}\otimes I_n\right)\Xi(\phi_{-1},1)\bigg)\,\mathrm{d}\phi_{-1}\nonumber\\&\qquad\qquad\qquad+\sum_{i=1}^{m+1}\kappa_i h_i (x(0,0))\nonumber
\end{align}
where
\begin{equation}
\Xi(\phi_{-1},t):=\begin{pmatrix}x(\phi_{1,1},\phi_{-1},t)\\\vdots\\x(\phi_{1,2N+1},\phi_{-1},t)\end{pmatrix},\,
F(\phi_{-1},t):=\begin{pmatrix}f(x(\phi_{1,1},\phi_{-1},t))\\\vdots\\f(x(\phi_{1,2N+1},\phi_{-1},t))\end{pmatrix}
\end{equation}
under variations in $\Xi$, $S$, and $\varrho$, when evaluated at $\Xi=G$, $S=T$, and $\varrho=\rho$.
\end{proposition}
\begin{proof}
The form of the discretization is an immediate consequence of the results of the preceding lemma and corollary. Variations of $L$ with respect to $\Xi$ yield the differential equations as well as the conditions
\begin{align}
0&=M^\mathsf{T}(\phi_{-1},1)-w^\mathsf{T}(\phi_{-1})\left(\mathcal{F}\otimes I_n\right)\\
0&=-M^\mathsf{T}(\phi_{-1},0)+w^\mathsf{T}(\phi_{-1}-\rho_{-1})\left(\mathcal{R}(\varrho_1)\mathcal{F}\otimes I_n\right)\nonumber\\
&\qquad+\delta_\mathrm{D}(\phi_{-1})\begin{pmatrix}1 & 0 & \cdots & 0\end{pmatrix}\otimes\sum_{i=1}^{m+1}\kappa_i\partial_xh_i(\gamma_1(0,0))
\end{align}
in terms of the Dirac delta function. Invertibility of the matrix in \eqref{eq:transversality} again implies that $\kappa_i=0$ for $i=1,\ldots,m$, these yield the discretized boundary conditions since $\mathcal{F}\mathcal{F}^\mathsf{T}=\mathcal{R}(\rho_1)\mathcal{F}\mathcal{F}^\mathsf{T}\mathcal{R}^\mathsf{T}(\rho_1)$.
\end{proof}

\begin{corollary}
The discretized form of the integral conditions obtained from Corollary~\ref{cor:lambdaphi} follows by replacing $\ln S$ in the expression for $L$ by $-\varrho_i$.
\end{corollary}

\bibliographystyle{plain}
\bibliography{AhsanDankowiczKuehn}
\end{document}